\documentclass[a4paper,12pt]{amsart}
\usepackage{amssymb,amsmath}
\let\amslrcorner\lrcorner
\usepackage{mathabx}
\usepackage{mathtools}
\usepackage{stmaryrd}
\usepackage{amscd,amsthm,amssymb}
\usepackage{enumerate}
\usepackage{color}
\usepackage[all,cmtip]{xy}
\usepackage{hyperref}
\usepackage{mathrsfs} 

\let\lrcorner\amslrcorner

\usepackage{latexsym}

\usepackage{hyperref}
\hypersetup{
    colorlinks=true, 
    linktoc=all,     
   linkcolor=blue,  
}

\def\sideremark#1{\ifvmode\leavevmode\fi\vadjust{\vbox to0pt{\vss
\hbox to 0pt{\hskip\hsize\hskip1em%
\vbox{\hsize2cm\tiny\raggedright\pretolerance10000%
\noindent {\color{red}{#1}}\hfill}\hss}\vbox to8pt{\vfil}\vss}}}%

\theoremstyle{plain}
\newtheorem{propn}{Proposition}[section]
\newtheorem{thm}[propn]{Theorem}
\newtheorem{lemma}[propn]{Lemma}
\newtheorem{cor}[propn]{Corollary}

\theoremstyle{definition}

\newtheorem{exo}[propn]{Example}

\newtheorem{rem}[propn]{Remark}

\def \bbR{\mathbb{R}}
\def\.{\cdot}

\def\d{{\mathrm d}}
\def\vs{\vskip .6cm}

\def\t{\tilde}
\def\beq{\begin{equation}}
\def\eeq{\end{equation}}
\def\bea{\begin{eqnarray*}}
\def\eea{\end{eqnarray*}}
\def\beaa{\begin{eqnarray}}
\def\eeaa{\end{eqnarray}}
\def\ba{\begin{array}}
\def\ea{\end{array}}
\def \L{\mathscr{L}}

\def \RM{\mathbb{R}}

\def \CM{\mathbb{C}}

\def\Ric{\mathrm{Ric}}
\def\id{\mathrm{id}}
\def\tr{\mathrm{tr}}
\def\Aut{\mathrm{Aut }}

\def\Hol{\mathrm{Hol}}
\def\Hom{\mathrm{Hom}}
\def\Sym{\mathrm{Sym}}

\def\su{\mathfrak{su}}

\def\sp{\mathfrak{sp}}
\def\gg{\mathfrak{g}}

\def\C{\mathscr{C}}
\def\D{\mathrm{D}}

\def\G{\mathrm{G}}
\def\H{\mathcal{H}}

\def\SO{\mathrm{SO}}

\def\End{\mathrm{End}}

\def\vol{\mathrm{vol}}
\def\Ker{\mathrm{Ker}}

\def\Sym{\mathrm{Sym}}
\def\scal{\mathrm{scal}}
\def\Id{\mathrm{id}}

\def\T{T}

\def\grad{\mathrm{grad}}

\def \t5{\frac{1}{\sqrt{5}}}
\def\g{\mathfrak{g}}
\DeclareMathOperator{\dH}{\d_{\mathcal{H}}}

\DeclareMathOperator{\di}{d}
\DeclareMathOperator{\spa}{span}

\DeclareMathOperator{\bbC}{\mathbb{C}}

\DeclareMathOperator{\p}{p}
\DeclareMathOperator{\q}{\mathbf{cn}}
\DeclareMathOperator{\diff}{Diff}
\DeclareMathOperator{\Hh}{\mathscr{H}}
\DeclareMathOperator{\Vv}{\mathscr{V}}
\DeclareMathOperator{\Sp}{Sp}
\DeclareMathOperator{\re}{Re}
\DeclareMathOperator{\IM}{Im}

\title{Eigenvalue estimates for $3$-Sasaki structures}
\author{Paul-Andi Nagy and Uwe Semmelmann}

\address{Paul-Andi Nagy\\
Center for Complex Geometry \\
Institute for Basic Science(IBS)\\
55 Expo-ro, Yuseong-gu \\
34126 Daejeon, South Korea
}
\email{paulandin@ibs.re.kr}

\address{Uwe Semmelmann\\
Institut f\"ur Geometrie und Topologie \\
Fachbereich Mathematik\\
Universit{\"a}t Stuttgart\\
Pfaffenwaldring 57 \\
70569 Stuttgart, Germany
}
\email{uwe.semmelmann@mathematik.uni-stuttgart.de}

\date{\today}

\begin{document}

\bibliographystyle{plain}

\begin{abstract}
We obtain new lower bounds for the first non-zero eigenvalue of the scalar sub-Laplacian for $3$-Sasaki metrics, improving 
the Lichnerowicz-Obata type estimates by Ivanov et al. in \cite{I2,I}. The limiting eigenspace is fully described in terms of the automorphism algebra. 
Our results can be thought of as an analogue of the Lichnerowicz-Matsushima estimate for K\"ahler-Einstein metrics. In dimension $7$, if the automorphism algebra is non-vanishing, we also compute the second eigenvalue for the sub-Laplacian and construct explicit eigenfunctions. In addition, for all metrics in the canonical variation of the $3$-Sasaki metric we give a lower bound for the spectrum of the Riemannian Laplace operator, depending only on scalar curvature and dimension. We also strengthen a result pertaining to the growth rate of harmonic functions, due to Conlon, Hein\&Sun \cite{CoH,HS}, in the case of hyperk\"ahler  cones. In this setup we also describe the space of holomorphic functions.

\vs



\end{abstract}
\maketitle
\tableofcontents
%
%
\section{Introduction} \label{intro}
%
%
Consider a compact manifold $M^{4n+3}$ equipped with a $3$-Sasaki structure $(g,\xi)$ consisting of a Riemannian metric $g$ and a triple $\xi=(\xi_1,\xi_2,\xi_3)$ of associated Reeb vector fields. In particular, the distribution $\Vv:=\spa\{\xi_1,\xi_2,\xi_3\}$ is tangent to a Riemannian foliation with totally geodesic leaves. The bracket generating distribution $\Hh:=\Vv^{\perp}$ has a natural sub-Riemannian structure which allows considering the scalar sub-Laplacian 
$\Delta_{\Hh}:C^{\infty}M \to C^{\infty}M
$ (see also Section \ref{3sas-p}). 
According to general results in sub-Riemannian geometry 
this operator is self-adjoint, non-negative and hypoelliptic and thus shares some of the properties of elliptic operators, including that of having discrete spectrum. See \cite{Stri,Rot} which are based on H\"ormander's 
criterion from \cite{Ho}; alternatively these properties follow from having $\Delta_{\Hh}$ coincide with the sub-Laplacian defined in \cite{I2,I} for the quaternion-contact structure induced by $(g,\xi)$, see also 
Remark \ref{coin} in this paper. Whenever $D:C^{\infty}M \to C^{\infty}M$ is a self-adjoint, non-negative operator admitting 
a discrete spectrum we denote with $\lambda_1(D)>0$ its first non-zero eigenvalue. 

In our previous work \cite{NS} we have observed that in dimension $7$ specific estimates for $\Delta_{\Hh}$ govern the deformation theory
of the second Einstein metric in the canonical variation of $g$. It is the aim of the present paper to bring this type of estimates to optimal form, in arbitrary dimension.

Note that Lichnerowicz-Obata type results for $\lambda_1(\Delta_{\Hh})$ are 
known in several more general set-ups by work of Baudoin et al:  when $\Vv$ is tangent to a totally geodesic Riemannian foliation, see \cite{BKW1}; when $\Hh$ is a sub-Riemannian structure of $H$ type, see \cite{BK}.
Both results work under the assumption that the transversal Ricci curvature is positive. Here by Lichnerowicz-Obata type we mean that equality is satisfied if and only if the sub-Riemannian structure on $\Hh$ can be extended to a metric of constant sectional curvature. 
See \cite{Li} and references therein for similar results for the Kohn Laplacian in CR geometry.
 

To state our first main result let $\mathfrak{g}:=\{X \in \Gamma(TM) : \L_X \xi^a=0\}$, where $\xi^a:=g(\xi_a,\cdot)$, be 
the automorphism Lie algebra of the $3$-Sasaki  structure. This is equipped with a tri-moment map $\mu=(\mu_1,\mu_2,\mu_3):\mathfrak{g} \to C^{\infty}(M,\bbR^3)$,
see \cite{BoGa} or Section \ref{tri-sectn} of the paper for definitions and main properties.  
\begin{thm} \label{main1}
Let $(M^{4n+3},g,\xi)$ be a compact $3$-Sasaki manifold. If $g$ does not have constant sectional curvature  the scalar sub-Laplacian satisfies the lower bound
\begin{equation*}\lambda_1(\Delta_{\Hh}) \geq 8n.
\end{equation*}
In addition, the limiting eigenspace
$\ker(\Delta_{\Hh}-8n)$
is isomorphic to $\gg \oplus \gg \oplus \gg $
where the isomorphism is induced by $(X_1, X_2,X_3) \in \gg \oplus \gg \oplus \gg \mapsto \sum \limits_{a=1}^3\mu_a(X_a).$
\end{thm}
Theorem \ref{main1} improves the Lichnerowicz-Obata-type lower bound $\Delta_{\mathscr{H}}\geq 4n$ proved by Ivanov et al in \cite{I2,I}. For, as showed in those works, the eigenspace $\ker(\Delta_{\mathscr{H}}-4n)$ vanishes unless $g$ has constant sectional curvature. Furthermore there is no shortage of structures 
$(g,\xi)$ with non-trivial automorphism algebra; in fact most known (non-homogeneous) examples are toric in the sense that $\mathfrak{g}$ contains the Lie algebra of a $(n+1)$-torus acting effectively on $M$. In such a situation we thus have 
$$
\lambda_1(\Delta_{\mathscr{H}})=8n.$$
\begin{rem} \label{rmk1-I}

\begin{itemize}
\item[(i)]
Theorem \ref{main1} is the $3$-Sasaki counterpart of a classical theorem of Lichnerowicz\&Matsushima which asserts that if $(Z,g_Z,J_Z)$ is K\"ahler Einstein with $\Ric^{g_Z}=Eg_{Z}$ and $E>0$ then 
$\Delta^{g_Z} \geq 2E$ on $\{f \in C^{\infty}Z : \int_Zf\vol=0\}$; the limiting eigenspace $\ker(\Delta^{g_Z}-2E)$ consists of Killing potentials and is thus isomorphic to $\mathfrak{aut}(Z,g)=\mathfrak{aut}(Z,J_Z)$ via $f \mapsto J_Z\grad\!f$.
\item[(ii)] The numerics in the lower bound in Theorem \ref{main1} differ from Lichnerowicz-Matsushima's estimate, for the $3$-Sasaki metric satisfies $\Ric^g=(4n+2)g$. 
\end{itemize}
\end{rem}
Remarkably, in dimension $7$, the techniques developed in this article also allow estimating the second eigenvalue of $\Delta_{\Hh}$ and investigate its geometric content.
\begin{thm} \label{main3}
Let $(M^{7},g,\xi)$ be a compact $3$-Sasaki manifold such that $g$ does not have constant sectional curvature and such that 
$\Aut(M,g)=\Aut(M,\xi) \times \SO(3)$. If $\gg \neq 0$, in addition to $\lambda_1(\Delta_{\Hh})=8$ we have 
$$ \lambda_2(\Delta_{\Hh})=16.
$$
The limiting eigenspace $\ker(\Delta_{\Hh}-16)$ contains the image of the map 
$$\Sym^2\gg \otimes \Sym^2_0\bbR^3 \to C^{\infty}M, \ \ (X \odot Y, \beta) \mapsto \beta(\mu(X),\mu(Y)).$$
\end{thm}
In Section \ref{tri-sectn} we show that this map is not zero on many sample elements. To explain the assumptions in Theorem  \ref{main3} recall that if $g$ does not have constant sectional curvature its isometry group enters the well known dichotomy 
\begin{equation} \label{dich}
\Aut(M,g)=\Aut(M,\xi) \times \Sp(1) \ \ \mbox{or} \ \ \Aut(M,g)=\Aut(M,\xi) \times \SO(3)
\end{equation}
where in each case the $\Sp(1)$ respectively $\SO(3)$ action is tangent to $\Vv$. Here the group of $3$-Sasaki automorphisms $\Aut(M,\xi):=\{f \in \diff(M) : f^{\star}\xi^a=\xi^a, 1 \leq a \leq 3\}$. In dimension $7$ both instances occur, with toric examples \cite{vanC2}. However, the generic case is when  $\SO(3) \subseteq \Aut(M,g)$ since many smooth examples are obtained by the Konishi construction for ASD, Einstein orbifolds, see \cite[Theorem 3.3.4]{BoGa1}.

The techniques used to establish Theorem \ref{main1} also yield the lower bound 
$$\lambda_1(\Delta^{g_T}) \geq 8(n+T^{-2})$$
where $\Delta^{g_T}$ is the scalar Riemannian Laplacian for the canonical variation $g_T, T>0$ of the $3$-Sasaki metric $g$, see Section \ref{crmk}. We believe this is new for $T \neq 1$; when $T=1$ the estimate can be derived, after a few additional arguments,
from \cite{CoH,HS}. Explaining how this works has led us to an improvement of the gap theorem in those references for the case of hyperk\"ahler cones. We also describe holomorphic functions on hyperk\"ahler cones. See Theorem \ref{gap} and Proposition \ref{hol-cone} in the paper. 
\subsection{Outline of the paper} \label{outline}
After describing some preliminary material on $3$-Sasaki structures and cone geometry we develop in 
Section \ref{prelm} the hyperk\"ahler cone version of an integral formula originally proved by Cheeger-Tian \cite{CT} for K\"ahler cones. This allows proving an estimate for a perturbation of the sub-Laplacian acting on horizontal $1$-forms, which has the advantage of being $\su(2)$ invariant. The $\su(2)$-representation theory on function spaces is brought into play in Section \ref{fctns} and used to show 
$\Delta_{\Hh}\geq 4nm$ on the weighted spaces $C_m^{\infty}M, m \in \mathbb{N}$ defined in Section \ref{lb1}. 
Secondly, we establish the refined estimate $\Delta_{\Hh} \geq (n+1)(4m+8)$ on $C^{\infty}_mM \cap \ker(\Delta_{\Hh}-4nm)^{\perp}$ where $m \in \mathbb{N}$. This is proved in Proposition \ref{est2} by using the integral formula mentioned above. 
Combining these estimates proves the overall lower bound for $\lambda_1(\Delta_{\Hh})$ in Theorem \ref{main1}. To determine the corresponding limiting space we first recall the well-known description of structure preserving Killing fields on Sasaki-Einstein structures in terms of their Killing potentials. The result is then proved in Section \ref{tri-sectn} via 
$\su(2)$ representation theoretical arguments. Theorem \ref{main3} is proved by combining again the above estimates for the sub-Laplacian and the following extra ingredients: when 
$\SO(3)\subseteq \Aut(M,g)$ the weighted spaces
$C_m^{\infty}M$ vanish when $m \in 2\mathbb{N}+1$; in dimension $7$, the first non-zero eigenvalue of the scalar basic Laplacian $\Delta_B:={\Delta_{\Hh}}_{\vert C_0^{\infty}M}$ satisfies  
$\lambda_1(\Delta_B)>16$, see \cite{LeeR} or Corollary \ref{basic1} in the paper. Section \ref{sectn4} deals with several function theoretic aspects 
of the hyperk\"ahler cone $(CM:=\bbR^{+} \times M, (\di\!r)^2+r^2g)$. In Theorem \ref{gap} we analyse the growth rate of homogeneous harmonic functions on the cone. The section ends with a description of holomorphic functions on the cone which gives a geometric interpretation of the limiting space $\ker(\Delta_{\Hh}-4nm) \cap C_m^{\infty}M$.

$\\$
{\bf{Acknowledgements:}} This research has been financially supported by the Special Priority Program
SPP 2026 `Geometry at Infinity' funded by the DFG. It is a pleasure to thank Craig van Coevering for many useful email exchanges and the reviewer for suggestions on how to improve presentation.
\section{Estimates on horizontal $1$-forms} \label{prelm}
\subsection{The $\su(2)$ action and horizontal operators in $3$-Sasaki geometry} \label{3sas-p}
%
%
First we recall some facts from $3$-Sasaki geometry which will be needed in what follows. 
Let $(M^{4n+3}, g)$ be a compact Riemannian manifold with a 3-Sasaki structure defined by a triple of  Killing vector fields $\xi=(\xi_1, \xi_2, \xi_3)$ satisfying $g(\xi_a,\xi_b)=\delta_{ab}$ 
and the $\su(2)$ commutator relations
\begin{equation} \label{su2}
\begin{split}
& [\xi_1,\xi_2]=2\xi_3, \ [\xi_2,\xi_3]=2\xi_1, \ [\xi_3,\xi_1]=2\xi_2. 
\end{split}
\end{equation}
The distributions $\Vv:=\spa \{\xi_1,\xi_2,\xi_3\}$ respectively $\Hh:=\Vv^{\perp}$ will be referred to as the vertical respectively the horizontal distributions. The vertical distribution 
induces a Riemannian foliation with totally geodesic leaves, denoted with 
$\mathcal{F}$, such that the leaf space $Q:=M\slash \mathcal{F}$ has the structure of a compact $4n$-dimensional quaternionic K\"ahler 
orbifold. The $3$-Sasaki metric is automatically Einstein with $\Ric^g = (4n+2)g$. The orbifold metric on $Q$ has scalar curvature $16n(n+2)$.

Throughout this paper the space of horizontal differential forms will be indicated with 
$\Omega^{\star}\Hh:=\{ \alpha \in \Omega^{\star}M : \xi_a \lrcorner \alpha=0 \ \mathrm{for} \ a=1,2,3\}$.

The differential geometric properties of $g$ are encoded in the structure equations for the coframe $\xi^a:=g(\xi_a, \cdot), a=1,2,3$ which read 
\begin{equation} \label{str-xi}
\di\!\xi^a=-2\xi^{bc}+2\omega_a
\end{equation}
with cyclic permutations on $abc$, where $\omega_1,\omega_2,\omega_3$ belong to $\Omega^2\!\Hh$. Here $\xi^{bc}=\xi^b \wedge \xi^c$ in shorthand notation. The triple $\{\omega_a, 1 \leq a \leq 3\}$ is algebraically constrained in the sense that 
$\omega_a=\omega_b(I_c \cdot, \cdot)$
with cyclic permutation on $abc$ where the almost complex structures $I_a:\Hh \to \Hh$ satisfy the quaternion relations 
$I_a \circ I_b=-I_b \circ I_a=I_c$.
The restriction $g_{\Hh}$ of the metric $g$ to $\Hh$ is thus recovered from 
$-\omega_a=g_{\Hh}(I_a \cdot, \cdot)$
with $1 \leq a \leq 3$. As a matter of convention, in the rest of the paper we let the almost complex structures $I_a$ act on $\Omega^1\!\Hh$ by composition, $I_a\alpha:=\alpha \circ I_a$. For proofs and more details the reader is referred  to \cite{BoGa}.

Next we describe some of the properties of the representation of $\su(2)$ on $\Omega^{\star}M$ given by the Lie derivatives $\L_{\xi_a}$ via the mapping 
$A_a \mapsto \L_{\xi_a}$, where $\{A_a, a=1,2,3\}$  is the  basis of $\su(2)$ given by
$$
A_1 = 
\begin{pmatrix*}
0 \; &  0 \; &  0 \\
0 \; &  0 \; &  \hbox to0pt{\hss$-$} 2\\
0 \; & 2 \; &  0
\end{pmatrix*}
\quad
A_2 = 
\begin{pmatrix*}
\;\;0                                  & 0  \;  &  2 \;    \\
\;\;0                                  & 0  \; &   0   \;   \\
\;\; \hbox to0pt{\hss$-$}2  & 0  \; &   0  \;
\end{pmatrix*}
\quad
A_3 = 
\begin{pmatrix*}
\;0  \;  &  \hbox to0pt{\hss$-$}2  &0 \;\\
\;2  \;  &                                    0&0 \;\\
\;0  \; &   0                                  &0\;
\end{pmatrix*}.
$$
Since $\xi_a$ are Killing vector fields preserving $\Hh$ we have $\L_{\xi_a}^{\star}=-\L_{\xi_a}$ on $\Omega^{\star}M$ respectively on $\Omega^{\star}\!\Hh$. 
Therefore the $\su(2)$-representation on $\Omega^{\star}M$ is orthogonal w.r.t the $L^2$-inner product induced by $g$ and preserves $\Omega^{\star}\!\Hh$.  The Casimir operator of the induced representation 
$\su(2) \times \Omega^{\star}\!\Hh \to \Omega^{\star}\!\Hh$ is defined as 
$$
\C:=-\sum \limits_{a}^{} \L^2_{\xi_a} :\Omega^{\star}\!\Hh \to \Omega^{\star}\!\Hh.
$$
This differs by a factor of $\tfrac18$ from the usual Lie theoretic definition involving the Killing form of $\su(2)$. The operator $\C$ is self-adjoint, 
non-negative and $\su(2)$-invariant. Yet another operator of relevance here is 
\begin{equation*}
\p:=\sum_a I_a \circ \L_{\xi_a} : \Omega^1\!\Hh \to \Omega^1\!\Hh.
\end{equation*}

It is clearly symmetric w.r.t. the $L^2$-inner product and enters the following
\begin{lemma} \label{l1s2}
The following hold on $\Omega^1\!\Hh$
\begin{itemize}
\item[(i)] $\p$ is $\su(2)$ invariant i.e. $[\p,\L_{\xi_a}]=0$
\item[(ii)] $\p \circ I_a+I_a \circ \p=4I_a-2\L_{\xi_a}$
\item[(iii)] $\p^2-2\!\p=\C$
\item[(iv)] $[\C-2\!\p,I_a]=0$.
\end{itemize}
\end{lemma}
\begin{proof}
Differentiating the structure equation  \eqref{str-xi} we obtain 
$\di\!\omega_a  =  2 \omega_b \wedge \xi^c - 2 \omega_c \wedge \xi^b$, 
in particular $\L_{\xi_a}\omega_b=-\L_{\xi_b}\omega_a=2\omega_c$ after using Cartan's formula. It follows that the Lie algebra $\su(2)$ acts 
on the triple $I_1,I_2,I_3$ according to 
\begin{equation} \label{Lie-osI}
\L_{\xi_a}I_b=-\L_{\xi_b}I_a=2I_c
\end{equation}
on $\Omega^1\!\Hh$. Together with \eqref{su2} this leads to the proof of all claims, after a few purely algebraic computations. See also Lemma 3.3 and Lemma 5.1 in \cite{NS} for detailed proofs in dimension $7$ which entirely carry through to arbitrary dimension.
\end{proof}

The horizontal differential operators of interest in this paper are mostly build around the horizontal exterior derivative which is defined according to 
$$ \d_{\Hh}:\Omega^{\star}\!\Hh \to \Omega^{\star+1}\!\Hh, \ \alpha \mapsto (\di\!\alpha)_{\Hh}$$
where the subscript above indicates projection onto  $\Omega^{\star}\!\Hh$ with respect to the splitting $\Omega^{\star}M=\Omega^{\star}\!\Vv \wedge \Omega^{\star}\!\Hh$. The latter notation is short for saying that 
$\Omega^{\star}M$ is spanned by elements of the form $\alpha \wedge \beta$ with $\alpha$ in $\Omega^{\star}\Vv$ respectively $\beta$ in $\Omega^{\star}\Hh$.

Cartan's formula shows that $\dH$ relates to the ordinary exterior derivative via 
\begin{equation} \label{Cart}
\di=\d_{\Hh}+\sum_{a} \xi^a \wedge \L_{\xi_a}.
\end{equation}
Note that the operators $\L_{\xi_a}$ preserve $\Omega^{\star}\!\Hh$ as $\Vv$ is totally geodesic. An important property of $\dH$ is its $\su(2)$-invariance, i.e.
$[\d_{\Hh},\L_{\xi_a}]=0$ which descends from the $\su(2)$ invariance of $\di$. 
The formal adjoint $\di_{\Hh}^{\star}:\Omega^{\star}\!\Hh \to \Omega^{\star-1}\!\Hh$ of $\d_{\Hh}$, computed w.r.t. the metric induced by $g_{\Hh}$, 
is also $\su(2)$-invariant i.e.  $[\di_{\Hh}^{\star},\L_{\xi_a}]=0$. 
These operators allow building the horizontal (or sub) Laplacian according to 
$$
\Delta_{\Hh}:=\di_{\Hh}\di_{\Hh}^{\star}+\di_{\Hh}^{\star}\di_{\Hh} : \Omega^{\star}\!\Hh \to \Omega^{\star}\!\Hh.
$$
This is self-adjoint, non-negative and hypoelliptic, in particular its spectrum is discrete. See e.g. 
\cite{BG} for a proof, based on H\"ormander's Theorem from \cite{Ho}, as well as Remark 
\ref{coin} below.
\begin{lemma} \label{s2l1}We have 
\begin{itemize}
\item[(i)]$\di^{\star}\!=\di_{\Hh}^{\star}$ on $\Omega^1\!\Hh$
\item[(ii)] $(\Delta^g \alpha)_{\Hh}= (\Delta_{\Hh} + \C)\alpha$ for $\alpha \in \Omega^1\!\Hh$
\item[(iii)]  $\Delta^g= \Delta_{\Hh} + \C $ on $C^{\infty}M$.
\end{itemize}
\end{lemma}
\begin{proof}
All claims are proved by easy $L^2$-orthogonality arguments based on \eqref{Cart}. See also \cite[Lemma 4.1]{NS} for a proof of (ii) and (iii) in dimension $7$ which extends automatically to arbitrary dimension. 
\end{proof}
In the remark below we show that when acting on functions, the operator $\Delta_{\Hh}$ coincides with the sub-Laplacian defined in \cite{I2,I}. This provides a direct way to see that the spectrum of $\Delta_{\Hh}$ is discrete.
\begin{rem} \label{coin}
Let $\nabla $ be the Biquard connection of the quaternionic-contact structure induced by the $3$-Sasaki structure on $M$. By equation (2.19) in \cite{I2}, the horizontal 
divergence operator is given by $\nabla^{\star}\sigma=-\sum (\nabla_{e_i}\sigma)e_i$ whenever $\sigma \in \Omega^1\Hh$ and where $\{e_i\}$ is a local orthonormal frame in $\Hh$. Using the comparison formula (2.2) between 
$\nabla$ and $\nabla^g$ in \cite{I2} shows that 
\begin{equation*}
\nabla^{\star}\sigma=-\sum (\nabla^g_{e_i}\sigma)e_i-\sum g(T(e_i, \sigma^{\sharp}),e_i)
\end{equation*}
where $T$ denotes the torsion tensor of the Biquard connection. However $T(\Hh,\Hh) \subseteq \Vv$ by \cite{I2}[Theorem 2.1] so the second summand above vanishes. Since $\sigma$ is horizontal and $\Vv$ is totally geodesic the first summand equals $\di^{\star_g}\sigma$. We conclude that 
$$\nabla^{\star}\sigma=\di^{\star_g}\sigma=\di^{\star}_{\Hh}\sigma$$ by using Lemma \ref{s2l1}, (i) for the last equality. Now we recall that the scalar sub-Laplacian $\Delta$ as defined in \cite{I2}, equation (3.1), acts on functions $f \in C^{\infty}M$ according to $\Delta f=-\tr^g_{\Hh}\nabla^2f=\nabla^{\star}\nabla f$; here $\nabla f$ is the horizontal gradient of $f$, hence $\nabla f=\di_{\Hh}f$ in our notation. It follows that $\Delta=\Delta_{\Hh}$ 
on functions, as claimed.
\end{rem}
\subsection{Elements of cone geometry} \label{cone-g}
An equivalent way to say that the metric $g$ admits a $3$-Sasaki structure is to require the metric cone 
$(CM:=M \times \bbR_{+},g_{\q}:=r^2g+(\di\!r)^2)$ of $M$ be hyperk\"ahler, in which case 
the triple of complex structures on $CM$ is determined from 
\begin{equation} \label{alg-J}
\begin{split}
&J_a\partial_r=-r^{-1}\xi_a, \ J_a\xi_b=\xi_c, \ J_a=I_a \ \mbox{on} \ \Hh 
\end{split}
\end{equation}
with cyclic permutations on $abc$. The corresponding symplectic forms, defined according to the convention  $\Omega^a=g_{\q}(J_a\cdot, \cdot)$, read 
$\Omega^a=-\tfrac{1}{2}\di(r^2\xi^a)$.

In view of subsequent use we also recall the comparison formulas between the form Laplacian $\Delta^{g_{\q}}$ on the cone and the form Laplacian $\Delta^g$ on the link $M$. Denoting 
$$\Omega^{\star}_rM=\{\alpha \in \Omega^{\star}CM : \partial_r \lrcorner \alpha=0\}$$ we consider the decomposition 
\begin{equation} \label{spl-c}
\Omega^{1}CM=\Omega^{1}_rM\oplus (\di\!r \wedge \Omega^{0}M).
\end{equation} 
Accordingly, we identify $1$-forms $\alpha+f\di\!r \in \Omega^1CM$ to vectors $\left (\begin{array}{c}
\alpha\\
 f \\
\end{array} \right ) $  $\in \Omega^1_rM \oplus \Omega^0_rM$. 
Following \cite[page 586]{C} or \cite[(2.14), (2.15)]{CT} we recall that with respect to this identification the codifferential $\di^{\star_{g_{\q}}}:\Omega^1CM \to C^{\infty}CM$ respectively the form Laplacian $\Delta^{g_{\q}} : \Omega^1CM \to  \Omega^1CM$ of the cone metric $g_{\q}$ have the matrix form 
\begin{equation} \label{co-di}
\di^{\star_{g_{\q}}}=\left ( \begin{array}{ccc}
&r^{-2}\di^{\star} & -\partial_r-(4n+3)r^{-1} 
\end{array} \right )
\end{equation}
respectively
\begin{equation} \label{lapl-c}
\Delta^{g_{\q}}=\left (\hspace{-3mm}\begin{array}{ccc}
&r^{-2}\Delta^g-\partial_r^2-(4n+1)r^{-1}\partial_r & -2r^{-1}\di \\
&-2r^{-3}\di^{\star} & r^{-2}\Delta^g-\partial_r^2+(4n+3)(-r^{-1}\partial_r+r^{-2})
\end{array} \hspace{-2mm}\right )
\end{equation}
where $\partial_r\alpha:=\L_{\partial_r}\alpha$ for forms $\alpha \in \Omega^{\star}_rM$. Similarly, on functions in $C^{\infty}CM$ we have
\begin{equation} \label{lapl-c0}
\Delta^{g_{\q}}=r^{-2}\Delta^g-\partial_r^2-(4n+3)r^{-1}\partial_r.
\end{equation}
\subsection{An integral formula of Cheeger-Tian type} \label{op-id}
The aim in this section is to obtain a lower bound for an explicit perturbation of $\Delta_{\Hh}: \Omega^{1}\!\Hh \to \Omega^{1}\!\Hh$. This 
will be achieved by using the hyperk\"ahler geometry of the metric cone of $(M,g)$. 
First we recall the available results of this type, starting with an integral formula due to Jeff Cheeger and Gang Tian \cite{CT}. To explain the details in that formula we recall that $(g,\xi_a),$ with $a \in \{1,2,3\}$, defines a Sasaki-Einstein structure on $M$ with corresponding horizontal distribution 
given by $\H_a=\spa \{\xi_b,\xi_c\} \oplus \Hh$  and almost complex structure $\tilde{I}_a:\H_a \to \H_a$ which acts according to $\tilde{I}_a\xi_b=\xi_c$ and $\tilde{I}_a=I_a$ on $\Hh$. In addition, whenever $X \in \Gamma(\H_a)$ we indicate with $\L_X^{\H_a}J_a$ the component of 
the Lie derivative $\L_XJ_a$ on $\End(\H_a) \subseteq \End(\T CM)$, w.r.t. the splitting $TCM=\spa\{\partial_r,\xi_a\} \oplus \H_a$. 
\begin{propn} \label{est-D}
Letting $\gamma=g(X,\cdot) \in \Omega^1\H_a$ we have 
\begin{equation*}
\frac{1}{2}\int_M \vert \L_X^{\H_a}J_a \vert^2_{g_{\q}}\vol=\int_M g(\Delta^g\gamma+\L^{2}_{\xi_a}\gamma+4(n+1)\tilde{I}_a\L_{\xi_a}\gamma-8(n+1)\gamma, \gamma) \vol.
\end{equation*}
\end{propn}
\begin{proof}
This is Equation 7A.48 in \cite[Corollary 7A.46]{CT}, used for the Ricci flat K\"ahler cone $(CM,g_{\q},J_a)$. Note that with respect to that reference $k=2n$ and the opposite sign on the term $\tilde{I}_a\L_{\xi_a}$ is explained by our convention for the action of $\tilde{I}_a$ on 
$\Omega^1\H_a$.
\end{proof}
In particular we obtain that  
the sub-Laplacian $\Delta_{\Hh}$ acting on $\Omega^1\!\Hh$ satisfies
\begin{equation} \label{prel-CTT}
\Delta_{\Hh}+4(n+1)I_a\L_{\xi_a}+2\L_{\xi_c}I_a\L_{\xi_b}-8(n+1)\geq 0
\end{equation} 
whenever $a \in \{1,2,3\}$, with cyclic permutations on $abc$.
For $a=1$ this is proved by integration after observing that $\vert \L_X^{\H_1}J_1 \vert^2_{g_{\q}}=\vert \L_X^{\Hh}J_1 \vert^2_{g_{\q}}+2 \vert (\L_XJ_1)\xi_2\vert^2$ where $\L_X^{\Hh}J_1$ is the component of $\L_XJ_1$ on $\End(\Hh) \subseteq \End(TCM)$.
However the operator in \eqref{prel-CTT} is not $\su(2)$-invariant, making it difficult to apply representation 
theoretical arguments. An estimate for an $\su(2)$-invariant perturbation of $\Delta_{\Hh}$ is obtained by summation 
in \eqref{prel-CTT} and reads 
\begin{equation*}
\Delta_{\Hh}+\frac{4n+6}{3}\p-8(n+1) \geq 0
\end{equation*}
on $\Omega^1\!\Hh$. The proof hinges on the algebraic identity $\p=\mathfrak{S}_{abc}\L_{\xi_c}I_a \L_{\xi_b}$ on $\Omega^1\Hh$ where 
$\mathfrak{S}_{abc}$ denotes cyclic summation on $abc$.

In order to be able to prove Theorem \ref{main1} the preliminary estimate for $\Delta_{\Hh}$ above must be considerably improved. To proceed in that direction we will use essentially the same cone technique as in \cite{CT} with two main differences: instead of the Lie derivative we use a suitable algebraic component of the exterior derivative on one forms. Secondly, the specific structure of the operators $\p$ and $\C$ must be taken into account. 

Consider the operator $\D: \Omega^1CM \to \Gamma(T^{\star}CM \otimes T^{\star}CM)$ given by 
$$ \D_U\alpha:=\nabla_U^{g_{\q}}\alpha+\sum_a \nabla_{J_aU}^{g_{\q}}J_a\alpha
$$
where $\nabla^{g_{\q}}$ indicates the Levi-Civita connection of $g_{\q}$. Here the endomorphisms $J_a$ act on 1-forms 
$\alpha \in \Omega^1CM$ by composition, $J_a\alpha:=\alpha \circ J_a$. The operator $\D$ features in the following Weitzenb\"ock-type formula which 
actually works on arbitrary hyperk\"ahler structures, i.e. not necessarily of cone type. 
\begin{lemma} \label{l23-ne} Letting $\alpha \in \Omega^1CM$ we have 
$$\frac{1}{4}\vert \D\alpha\vert^2_{g_{\q}}=g_{\q}(\Delta^{g_{\q}}\alpha,\alpha)-\frac{1}{2}\Delta^{g_{\q}}g_{\q}(\alpha,\alpha)+\sum_a \di^{\star_{g_{\q}}}\! \eta_a(\alpha)$$
where $\eta_a(\alpha) \in \Omega^1CM$ is defined according to $\eta_a(\alpha)(U):=g_{\q}(\nabla^{g_{\q}}_{J_aU}\alpha,J_a\alpha)$.
\end{lemma}
\begin{proof}
Although the proof is fairly elementary we give some details for the convenience of the reader. Let $\{e_i\}$ be a local orthonormal basis in $(CM,g_{\q})$. From the definitions 
\begin{equation*} 
\begin{split}
\di^{\star_{g_{\q}}}\!\eta_a(\alpha)=&-\sum_i g_{\q}((\nabla^{g_{\q}})^2_{e_i,J_a e_i}\alpha,J_a\alpha) -\sum_ig_{\q}(\nabla^{g_{\q}}_{J_ae_i}\alpha,\nabla^{g_{\q}}_{e_i}(J_a\alpha))\\
=&-\sum_ig_{\q}(\nabla^{g_{\q}}_{J_ae_i}\alpha,\nabla^{g_{\q}}_{e_i}(J_a\alpha))
\end{split}
\end{equation*} 
by also taking into account the Ricci identity on $\Omega^1CM$ and that $g_{\q}$ is Ricci flat. At the same time purely algebraic arguments  show that  
\begin{equation*}
\sum_i g_{\q}(J_a \nabla^{g_{\q}}_{J_ae_i}\alpha,J_b\nabla^{g_{\q}}_{J_be_i}\alpha)=\sum_ig_{\q}(\nabla^{g_{\q}}_{e_i}\alpha,J_c\nabla^{g_{\q}}_{J_ce_i}\alpha)=
\di^{\star_{g_{\q}}}\eta_{c}(\alpha)
\end{equation*}
with cyclic permutation on the indices $abc$. Expansion of the norm of $\D\alpha$ based on this formula 
thus yields 
$\frac{1}{4}\vert \D\alpha\vert^2_{g_{\q}}=\vert \nabla^{g_{\q}} \alpha \vert^2_{g_{\q}}+\sum_a \di^{\star_{g_{\q}}}\! \eta_a(\alpha)$. As $g_{\q}$ is Ricci flat the Bochner formula on $\Omega^1CM$ reads $\Delta^{g_{\q}}=(\nabla^{g_{\q}})^{\star}\nabla^{g_{\q}}$ hence 
the claim follows from the identity 
$\tfrac{1}{2}\Delta^{g_{\q}}g_{\q}(\alpha,\alpha)=g_{\q}((\nabla^{g_{\q}})^{\star}\nabla^{g_{\q}}\alpha,\alpha)-\vert \nabla^{g_{\q}}\alpha \vert^2_{g_{\q}}.$
\end{proof}
The main idea in what follows is to integrate the Weitzenb\"ock-type formula from Lemma \ref{l23-ne} in direction of 
$M$.
To do so we first need to compute the integrals of the divergence type terms on $CM$ in Lemma \ref{l23-ne} in the special case of horizontal forms with quadratic growth.
\begin{lemma} \label{div}
Consider $\alpha:=r^2\gamma \in \Omega^1CM$ where $\gamma \in \Omega^1\!\Hh$. The following hold
\begin{itemize}
\item[(i)] $\int_M\Delta^{g_{\q}}g_{\q}(\alpha,\alpha)\vol=-8(n+1)\int_M g(\gamma,\gamma)\vol$ \\[-1.5ex]
\item[(ii)] $\int_M \di^{\star_{g_{\q}}}\eta_a(\alpha)\vol=-4(n+1)\int_M g(I_a\L_{\xi_a}\gamma-\gamma,\gamma)\vol$ \ for $a \in \{1,2,3\}$.
\end{itemize}
\end{lemma}
\begin{proof}
(i) We have 
$ \Delta^{g_{\q}}g_{\q}(\alpha,\alpha)=\Delta^gg(\gamma,\gamma)-8(n+1)g(\gamma,\gamma)
$
by the comparison formula \eqref{lapl-c0} 
and the claim follows by integration over $M$.\\
(ii) Split $\eta_a(\alpha)=\gamma_1+\gamma_0 \di\!r$ according to 
$\Omega^1CM=\Omega^1_rCM \oplus \di\!r \wedge C^{\infty}CM$. From the definitions 
\begin{equation*} \label{L-co4}
\gamma_0=\eta_a(r^2\gamma)\partial_r=-r^{-1}g_{\q}(\nabla^{g_{\q}}_{\xi_a}(r^2\gamma),J_a(r^2\gamma))=-rg(\nabla^g_{\xi_a}\gamma,I_a\gamma)
\end{equation*}
by taking into account that $\gamma \in \Omega^1\!\Hh$. At the same time 
\begin{equation} \label{na-xi}
\nabla^g_{\xi_a}\gamma=\L_{\xi_a}\gamma+I_a\gamma
\end{equation} 
as a consequence of the structure equations \eqref{str-xi}. Summarising, 
$ \gamma_0=rg(I_a\L_{\xi_a}\gamma-\gamma,\gamma).$
The comparison formula in \eqref{co-di} and Stokes' theorem thus yield 
\begin{equation*}  
\begin{split}
\int_M \di^{\star_{g_{\q}}}\eta_a(\alpha)\vol=&\quad \int_M( r^{-2}\di^{\star}\gamma_1-\partial_r \gamma_0-(4n+3)r^{-1}\gamma_0)\vol
\\
=&-\int_M (\partial_r \gamma_0+(4n+3)r^{-1}\gamma_0)\vol=-(4n+4)\int_Mg(I_a\L_{\xi_a}\gamma-\gamma,\gamma)\vol
\end{split}
\end{equation*}
as claimed.
\end{proof}
Next we need to gain insight into the contribution of vertical terms to norms of type $\vert \D\alpha\vert^2_{g_{\q}}$ with 
$\alpha$ as in Lemma \ref{div}. To that extent we first let $\Omega^1_r\!\Hh$ be the $g_{\q}$-orthogonal complement of $\spa \{\di\!r, \xi^1,\xi^2,\xi^3\}$ within $\Omega^1CM$. Then \eqref{spl-c} yields the 
$g_{\q}$-orthogonal splitting 
$\Omega^1CM=\spa \{\di\!r, \xi^1,\xi^2,\xi^3\} \oplus \Omega^1_r\!\Hh$.
With respect to this splitting we define 
\begin{equation} \label{DHd}
\D^{\Hh}_U\alpha:=(\D_U\alpha)_{\Omega^1_r\!\Hh}
\end{equation}
whenever $(U,\alpha) 
\in TCM \times \Omega^1CM$.
\begin{lemma} \label{no-ver}
Assuming that $\alpha=r^2\gamma$ with $\gamma \in \Omega^1\!\Hh$ we have 
\begin{itemize}
\item[(i)] $\D_{\partial_r}\alpha=r(4-\p)\gamma$ \\[-2ex]
\item[(ii)] $\int_M \vert \D_{\partial_r}\alpha\vert^2_{g_{\q}}\vol=
\int_M g((\C-6\!\p+16)\gamma,\gamma)\vol$\\[-2ex]
\item[(iii)] $\D_X\alpha=\D^{\Hh}_X\alpha+2r^{-1}\alpha(X)\di\!r+2\sum_{a}\alpha(I_aX)\xi^a$ \ \ with $X \in \Hh$.
\end{itemize}
\end{lemma}
\begin{proof}
(i) From the algebraic definition of the complex structures $J_a$ in \eqref{alg-J} we get 
$\D_{\partial_r}\alpha=\nabla^{g_{\q}}_{\partial_r}\alpha-r\sum_a J_a\nabla_{\xi_a}^{g_{\q}}\gamma.$
The claim follows by using the identities $\nabla^{g_{\q}}_{\partial_r}=\partial_r-r^{-1}$ on $\Omega^1M \subseteq \Omega^1_rCM$ and $\nabla_{\xi_a}^{g_{\q}}=\nabla^g_{\xi_a}=\L_{\xi_a}+I_a$ on $\Omega^1\!\Hh$.\\
(ii) follows from (i) and part (iii) in Lemma \ref{l1s2} since the operator $\p$ is self-adjoint.\\
(iii) Recall that $\nabla_X^{g_{\q}}\partial_r=r^{-1}X$ and hence $\nabla_X^{g_{\q}}\xi_a=-I_aX$ since $(g_{\q},J_a)$ are K\"ahler. As
$\alpha$ vanishes on the distribution $\spa \{\partial_r,\xi_1,\xi_2,\xi_3\}$ these facts lead to 
\begin{equation*}
\nabla_X^{g_{\q}}\alpha=\nabla^{\Hh}_X\alpha-r^{-1}\alpha(X)\di\!r+\sum_a\alpha(I_aX) \xi^a
\end{equation*}
where $\nabla^{\Hh}_X\alpha$ is the projection of $\nabla_X^{g_{\q}}\alpha$ onto $\Omega^1_r\!\Hh \subset \Omega^1CM$.
After using the variable change $(X,\alpha) \mapsto (J_aX,J_a\alpha)$ this yields 
$$\nabla_{J_aX}^{g_{\q}}J_a\alpha=\nabla^{\Hh}_{I_aX}I_a\alpha+r^{-1}\alpha(X)\di\!r-2\alpha(I_aX)\xi^a+\sum_b\alpha(I_bX) \xi^b$$ and the claim follows from the definition of $\D$ by gathering terms.
\end{proof}
Thus prepared, the main estimate for the sub-Laplacian $\Delta_{\Hh}$ acting on horizontal $1$-forms needed in this paper is now at hand.
\begin{propn} \label{Lco-2}
The sub-Laplacian $\Delta_{\Hh}$ acting on $\Omega^1\!\Hh$ satisfies 
\begin{equation*}
\Delta_{\Hh} \geq (4n-2)\p-8n+8
\end{equation*}
with equality on the space $\{\gamma \in \Omega^1\!\Hh : \D^{\Hh}_X\gamma=0, X \in \Hh\}$, where the operator 
$\D^{\Hh}$ is defined as in \eqref{DHd}.
\end{propn}
\begin{proof}
The first step is to integrate the norm identity from Lemma \ref{l23-ne} with $\alpha=r^2\gamma$. Take into account Lemma \ref{div} 
and that $g_{\q}(\Delta^{g_{\q}}\alpha,\alpha)=g((\Delta^g-2(4n+2))\gamma,\gamma)$ as granted by \eqref{lapl-c}; after a short algebraic computation we end up with 
\begin{equation} \label{prel-1e}
\frac{1}{4}\int_M \vert \D \alpha \vert^2_{g_{\q}}\vol=\int_Mg((\Delta^g-4(n+1)\p+8n+12)\gamma,\gamma)\vol.
\end{equation}
To prove the claim there remains to estimate $\vert \D \alpha \vert^2_{g_{\q}}$ as follows. The definition of the operator $\D$ leads, after a short computation, to $J_a\D_{J_aU}=\D_U$ on $\Omega^1CM$ for $a \in \{1,2,3\}$ and $U \in TCM$; in particular $r^{-2}\vert \D_{\xi_a}\alpha \vert^2_{g_{\q}}=\vert \D_{\partial_r} \alpha \vert^2_{g_{\q}}$. Letting $\{e_i\}$ be a local $g$-orthonormal frame in $\Hh$ we thus get 
\begin{equation*}
\begin{split}
\vert \D \alpha \vert^2_{g_{\q}}=&\vert \D_{\partial_r} \alpha \vert^2_{g_{\q}}+r^{-2}\sum_a \vert \D_{\xi_a}\alpha \vert^2_{g_{\q}}+ r^{-2}\sum_i \vert \D_{e_i}\alpha \vert^2_{g_{\q}}\\
=&4\vert \D_{\partial_r} \alpha \vert^2_{g_{\q}}+r^{-2}\sum_i \vert \D_{e_i}\alpha \vert^2_{g_{\q}}=
4\vert \D_{\partial_r} \alpha \vert^2_{g_{\q}}+16g(\gamma,\gamma)+\vert \D^{\Hh} \alpha \vert^2_{g_{\q}}
\end{split}
\end{equation*}
after taking into account part (iii) in Lemma \ref{no-ver}. Plugged into \eqref{prel-1e} this yields 

\begin{equation*} 
\begin{split}
\frac{1}{4}\int_M\vert \D^{\Hh} \alpha \vert^2_{g_{\q}}\vol=&\int_Mg((\Delta^g-4(n+1)\p+8n+8)\gamma,\gamma)\vol
-\int_M\vert \D_{\partial_r} \alpha \vert^2_{g_{\q}} \vol\\
=&\int_Mg((\Delta^g-\C+(-4n+2)\p+8n-8)\gamma,\gamma)\vol\\
=& \int_Mg((\Delta_{\Hh}+(-4n+2)\p+8n-8)\gamma,\gamma)\vol
\end{split}
\end{equation*}
after successively using part (ii) in Lemma \ref{no-ver} and $g(\Delta^g \gamma, \gamma)= g(\Delta_{\Hh}\gamma + \C \gamma, \gamma)$ as granted by Lemma \ref{s2l1}, (ii). The estimate in the claim is therefore proved, with equality 
on forms $\gamma \in \Omega^1\!\Hh$ such that $\D^{\Hh}_X\gamma=0, X \in \Hh$. 
\end{proof}
The structure of the limiting space $\{\gamma \in \Omega^1\!\Hh : \D^{\Hh}_X\gamma=0, X \in \Hh\}$ will not be further explored, as it is not needed in this paper.
\section{Estimates on functions} \label{fctns}
\subsection{Lower bounds for $\Delta_{\Hh}$} \label{lb1}
Recall that the (real) irreducible finite dimensional representations of the Lie algebra $\su(2)$ are completely determined by their 
dimension or their Casimir eigenvalues. They come in two series:
$$
(\pi_l,U_l)_{l \geq 1} \quad \mbox{with}\quad  \dim_{\bbR}U_l=2l+1 \qquad  \mbox{and} \qquad  (\rho_l,V_{l})_{l \in \mathbb{N}} \quad  \mbox{with} \quad  \dim_{\bbR}V_l=4l+4.
$$
Their explicit realisation is not needed here, we only record that the Casimir operator $\C$ acts on 
$U_l$ respectively $V_l$ as $m(m+2)$ with $m=2l$ respectively $m=2l+1$. From these purely algebraic facts we derive the following lemma which will be used frequently in this paper.
\begin{lemma}\label{dense} The following hold
\begin{itemize}
\item[(i)] the eigenvalues of the Casimir operator $\C$ acting on $C^{\infty}M$ are of the form 
$m(m+2)$ with $m \in \mathbb{N}$
\item[(ii)] the direct sum of 
the $\mathfrak{su}(2)$ invariant spaces 
$$
C^{\infty}_m M:=C^{\infty}M \cap \ker(\C-m(m+2)), m \in \mathbb{N}
$$ 
is dense in $L^2M$.
\end{itemize}
\end{lemma} 
\begin{proof}
This is a standard argument based on the invariance properties of $\Delta^g$ which goes as follows.
Since $\xi_a,a=1,2,3$ are Killing vector fields, any eigenspace of $\Delta^g$, say $E_{\lambda}$, is preserved by the Lie derivatives $\L_{\xi_a}, a=1,2,3$ and thus becomes a finite dimensional $\su(2)$-representation. Splitting $E_{\lambda}$ into irreducible summands, counted with multiplicities, shows that $\C$ acts on any such irreducible component by multiplication with 
$m(m+2)$ for some $m \in \mathbb{N}$. This is due to the structure of $\su(2)$-representation recalled above. 
Both claims follow now, after re-ordering of the $\su(2)$-irreducible summands, from having the direct sum of eigenspaces of $\Delta^g$ dense in $L^2M$. 
\end{proof}
Note that $C^{\infty}_0 M = C^{\infty}_{B} M$ is the space of $\su(2)$ invariant or 
basic functions and that the spaces $C^{\infty}_mM$ are infinite dimensional for any 
$m \in \mathbb{N}$, see also Remark \ref{reg-case}.
To further setup notation, whenever $(\pi,V)$ is a finite 
dimensional real representation of  $\su(2)$ we write
\begin{equation*}
C^{\infty}_{\pi}(M,V):=\{f \in C^{\infty}(M,V) : \L_{\xi_a}f=-\pi(A_a) \circ f, a=1,2,3\}  \ 
\end{equation*}
for the basis $\{A_1,A_2,A_3\}$ of  $\su(2)$ considered in Section \ref{3sas-p}. These functions spaces will be mainly used for the adjoint representation $(\pi_1,U_1)$ when $U_1=\bbR^3$ and $\pi_1$ acts by matrix multiplication on the basis $A_1,A_2,A_3$ of $\su(2)$. Nevertheless we observe that, more generally, 
\begin{lemma} \label{peter-w}
The map \begin{equation} \label{embd1}
\Hom_{\bbR}(U_l,C^{\infty}_{\pi_l}(M,U_l))\to C^{\infty}_{2l}M, \ F \mapsto (x \in M \mapsto \tr (v \mapsto F(v)x))
\end{equation}
is well defined and surjective. 
\end{lemma}
\begin{proof}
Let $f \in C^{\infty}_{\pi_l}(M,U_l)$; from the definitions it is clear that the component functions of 
$f$, with respect to a basis in $U_l$, belong to $C^{\infty}_{2l}(M)$. Thus the map in the statement is well 
defined. Now we prove that the map in \eqref{embd1} is surjective. Using eigenspaces of $\Delta^g$, as in the proof of Lemma \ref{dense}, it is enough to consider finite dimensional $\su(2)$ subrepresentations in $C^{\infty}_{2l}M$. As the Casimir operator acts as $2l(2l+2)$ these are isomorphic to a direct sum of representations isomorphic to $U_l$. By linearity we may thus reduce to prove surjectivity for just one such summand, say $E \subseteq C^{\infty}_{2l}M$. Choose an $\su(2)$-equivariant isomorphism $I:E \to U_l$, that is $I(\L_{\xi_a}f)=\pi_l(A_a) \circ I(f)$ for all $f \in E$ and 
$a=1,2,3$. We may assume that $I$ is an isometry, when $E$ is equipped with the $L^2$ scalar product and 
$\langle \cdot, \cdot \rangle$ is a  $\su(2)$-equivariant 
inner product  on $U_l$. Also choose a $L^2$-orthonormal basis of functions $\{f_1, \ldots, f_{2l+1}\}$ in $E$ and 
consider the linear basis $\{\mathbf{e}_1:=If_1, \ldots, \mathbf{e}_{2l+1}:=If_{2l+1} \}$ in $U_l$; the latter is thus orthonormal with respect to the inner product $\langle \cdot, \cdot \rangle$ on $U_l$. Having the isomorphism $I$ equivariant under $\su(2)$ thus reads 
$\L_{\xi_a}f_i=\sum_k \langle \pi_l(A_a)\mathbf{e}_i,\mathbf{e}_k \rangle f_k.$
Now consider $\phi:=\sum_i f_i \mathbf{e}_i$ in $C^{\infty}(M,U_l)$; then 
\begin{equation*}
\begin{split}
\L_{\xi_a} \phi&=\sum_i (\L_{\xi_a}f_i) \mathbf{e}_i=\sum_{i,k} \langle \pi_l(A_a)\mathbf{e}_i,\mathbf{e}_k \rangle f_k \mathbf{e}_i=-\sum_{i,k} \langle e_i, \pi_l(A_a)\mathbf{e}_k \rangle f_k\mathbf{e}_i\\
&=-\sum_kf_k\pi_l(A_a)\mathbf{e}_k=-\pi_l(A_a) \circ \phi
\end{split}
\end{equation*}
since $\langle \cdot, \cdot \rangle$ is $\su(2)$-invariant.
We have thus showed that $\phi$ belongs to $C^{\infty}_{\pi_l}(M,U_l)$.
Finally, defining $F \in \Hom_{\bbR}(U_l,C^{\infty}_{\pi_l}(M,U_l))$ on basis vectors via $F\mathbf{e}_i=\lambda_i \phi$ it follows that the image of the map 
in \eqref{embd1} on $F$ is $\sum_i \lambda_i f_i$ for any 
real coefficients $\lambda_1, \ldots, \lambda_{2l+1}$ and the surjectivity onto $E$ is proved.
\end{proof}
An entirely similar statement can be made for $C^{\infty}_{2l+1}M$.
To be able to prove spectral estimates for $\Delta_{\Hh}$ we also need a few operator identities. 
\begin{lemma} \label{h-diff} The following hold for  $f \in C^{\infty}M$
\begin{itemize}
\item[(i)] \label{dId}$\di_{\Hh}^{\star}I_a\d_{\Hh}\!f=-4n\L_{\xi_a}f$ 
\item[(ii)]\label{dpd} $\di_{\Hh}^{\star} \p \di_{\Hh}f=4n\C f$
\item[(iii)] \label{comdd}$[\Delta_{\Hh}, \, \d_{\Hh} ]f = 2\p(\d_{\Hh}\!f).$ 
\end{itemize}
\end{lemma}
\begin{proof}
(i) Because $\di^{\star_{g_{\q}}}(f\omega_1)=-J_1\di\!f$ with respect to the K\"ahler structure $(g_{\q},J_1)$ 
we obtain 
$\di^{\star_{g_{\q}}}\!J_1\di f=0$. At the same time expanding $\di\!f$ according to \eqref{Cart} yields 
$$ J_1 \di\!f=I_1\d_{\Hh}\!f+\L_{\xi_3}f\xi^2-\L_{\xi_2}f\xi^3-\di\!r \wedge r^{-1}\L_{\xi_1}f$$
since $J_1\xi^1=-r^{-1}\di\!r, \ J_1\xi^2=-\xi^3, J_1\xi^3=\xi^2$. Taking \eqref{co-di} into account thus 
leads to 
$\di^{\star}(I_1\dH\!f+\L_{\xi_3}f\xi^2-\L_{\xi_2}f\xi^3)+(4n+2)\L_{\xi_1}f=0$. As $\xi^2$ and $\xi^3$ are co-closed we compute further 
$\di^{\star}(\L_{\xi_3}f\xi^2-\L_{\xi_2}f\xi^3)=-\L_{\xi_2}\L_{\xi_3}f+\L_{\xi_3}\L_{\xi_2}f=-2\L_{\xi_1}f$ and the claim follows from having $\di^{\star}=\di_{\Hh}^{\star}$ on $\Omega^1\!\Hh$. \\
(ii) follows from (i) and the $\su(2)$ invariance of $\d_{\Hh}$ respectively $\di_{\Hh}^{\star}$.\\
(iii) From $\Delta^g(\di\!f)=\di\!\Delta^gf$ and \eqref{Cart} we get 
\begin{equation} \label{prel-pp}
\Delta^g(\d_{\Hh}\!f)+\sum_a\Delta^g(\L_{\xi_a}f\xi^a)=\di\!\Delta^gf.
\end{equation}
As $\xi_a$ are Killing fields and $\Ric^g=(4n+2)g$ calculation using the Bochner formula $\Delta^g=(\nabla^g)^{\star}\nabla^g+\Ric^g$ shows that 
$$\Delta^g(\L_{\xi_a}f\xi^a)=((\Delta^g+8n+4)\L_{\xi_a}f)\xi^a-\grad \L_{\xi_a}f \lrcorner \di\!\xi^a.$$ In particular 
the horizontal projection $(\Delta^g(\L_{\xi_a}f\xi^a))_{\Hh}=-2I_a\d_{\Hh}\!\L_{\xi_a}f$ after using the structure equations \eqref{str-xi}. The claim follows now easily by projecting \eqref{prel-pp} onto $\Omega^1\!\Hh$ whilst taking into account that $(\Delta^g(\d_{\Hh}\!f))_{\Hh}=(\Delta_{\Hh}+\C)\d_{\Hh}\!f$ and $\Delta^gf=\Delta_{\Hh}f+\C f$, as granted by Lemma \ref{s2l1},(ii)\&(iii).
\end{proof}
Based on these preliminaries we can now prove the following 
\begin{propn} \label{est-f1}
Let $(M^{4n+3},g,\xi)$ be a compact $3$-Sasaki manifold. The sub-Laplacian $\Delta_{\Hh}$ acting on $C_m^{\infty}M$ where  $m \geq 1$ satisfies the lower bound 
$$\Delta_{\Hh} \ge  4nm.$$  
For the limiting eigenspaces we have 
\begin{itemize}
\item[(i)] $\ker(\Delta_{\Hh}-4nm)\cap C_m^{\infty}M=\{f \in C_m^{\infty}M : \p(\d_{\Hh}\!f)=(m+2)\d_{\Hh}\!f\}$
\item[(ii)]$\ker(\Delta_{\Hh}-4n) \cap C_1^{\infty}M=0$ if $g$ does not have constant sectional curvature.
\end{itemize}
\end{propn}
\begin{proof}
Pick $f \in C^\infty_m M$ with $\Delta_{\Hh} f = \nu f$ and compute the $L^2$ norm
$$
\| \p(\d_{\Hh} f)  - (m+2)  \d_{\Hh} f \|^2_{L^2} \;=\; \| \p(\d_{\Hh} f) \|^2_{L^2} \;+\;  (m+2)^2 \|  \d_{\Hh} f  \|^2_{L^2} \;-\; 2(m+2) \langle \p(\d_{\Hh} f), \d_{\Hh} f\rangle_{L^2} 
$$
\vspace{-.6cm}
\begin{eqnarray*}
&=& \langle \p^2(\d_{\Hh} f),  \d_{\Hh} f \rangle_{L^2} \;+\; (m+2)^2 \nu \|   f  \|^2_{L^2} \;-\; 8n(m+2) \langle \C f, f \rangle_{L^2} \\[1ex]
&=& \langle \C (\d_{\Hh} f),  \d_{\Hh} f \rangle_{L^2}+2\langle \p(\d_{\Hh} f), \d_{\Hh} f\rangle_{L^2}+(m+2)^2 \nu \|   f  \|^2_{L^2} \;-\; 8n(m+2) \langle \C f, f \rangle_{L^2} \\[1ex]
&=& [m(m+2)\nu + 8nm(m+2) + (m+2)^2\nu - 8nm(m+2)^2] \Vert f \Vert^2_{L^2} \\[1ex]
&=& 2(m+1)(m+2)(\nu - 4nm )\|   f  \|^2_{L^2} \ .
\end{eqnarray*}
This calculation is mainly based on the identities $\p^2-2\p=\C$ on $\Omega^1\!\Hh$ as well as $\di_{\Hh}^{\star}\p\di_{\Hh}=4n\C$ on
$C^{\infty}M$ from Lemma \ref{l1s2},(iii) respectively Lemma \ref{h-diff},(ii). By positivity, $\nu \geq 4nm$ and $\p(\d_{\Hh}\!f)=(m+2)\d_{\Hh}\!f$ in case of equality. As 
$\Delta^g=\Delta_{\Hh}+\C$ on functions, part (ii)  follows from $\ker(\Delta_{\Hh}-4n) \cap C_1^{\infty}M \subseteq \ker(\Delta^g-(4n+3))$ and Obata's theorem for the Einstein metric $g$.
\end{proof}
As explained next, the estimate in Proposition \ref{est-f1} generalises some results previously only known for quaternion K\"ahler geometry. 
\begin{rem} \label{reg-case}
Assume that the $3$-Sasaki manifold is regular, i.e. 
$(M^{4n+3}, g)$  fibres as a principal $\mathrm{Sp}(1)$ or $\mathrm{SO}(3)$ bundle over a  quaternion K\"ahler manifold
$(Q^{4n}, g_Q)$. In this situation we explain how the estimate in Proposition \ref{est-f1}, with $m=2l$, can be derived from previous work \cite{SW,gu1}. 
First we record the well-known identification 
\begin{equation} \label{idl}
C^\infty_{\pi_l}(M, U_l\otimes \CM) =C^\infty_{\pi_l}(M, \Sym^{2l} \CM^2) \cong \Gamma(\Sym^{2l} H) 
\end{equation}
where $H$ is the locally defined vector bundle on the quaternion K\"ahler manifold $Q$ associated to the standard 
representation of $\mathrm{Sp}(1)$ on $\CM^2$. Recall that the bundle $\Sym^{2l} H$ is globally defined on $Q$.
Under \eqref{idl} the horizontal Laplacian $\Delta_{\Hh} = \Delta^g - \C$ on functions in  
$C^\infty_{\pi_l}(M, U_l\otimes \CM) $ corresponds to $\nabla^{\star}\nabla$ on sections of $\Sym^{2l} H$ (see \cite[Proposition 5.6]{BGV}). 
Here $\nabla$ is the connection induced by the Levi-Civita connection. On sections of $\Sym^{2l} H$ we also have the standard Laplacian  $\Delta_{\pi_l} := \nabla^{\star}\nabla + q(R)$, where 
$q(R)$ is a curvature term which acts as  $\tfrac{2l(1+l)\scal_{g_Q}}{4n(n+1)}\,  \Id = 8l(l+1) \id$ 
(see \cite{SW}[Lemma 3.4]); note that the definition of $q(R)$ here differs by a factor of $2$ compared to that reference. The estimate 
$
\Delta_{\pi_l} \ge \tfrac{l(n+1+l)}{2n(n+2)} \scal_{g_Q} = 8l(n+1+l)
$ 
proved in \cite[Lemma 3.1]{gu1} thus leads to $
\Delta_{\Hh}  \ge 8l(n+1 +l) -8l(l+1) = 8nl$
on $C^\infty_{\pi_l}(M, U_l \otimes \bbC)$. The same estimate for $\Delta_{\Hh}$ acting on scalar valued functions in $C_{2l}^{\infty}M$ can be now recovered with the aid of the surjection in \eqref{embd1}.
\end{rem}

In order to prove Theorem \ref{main1} we need a refined lower bound for $\Delta_{\Hh}$ that is a lower bound for $\Delta_{\Hh}$ acting on the $L^2$ orthogonal complement of 
$\ker(\Delta_{\Hh}-4nm)\cap C_m^{\infty}M$ within $C_m^{\infty}M$. To produce such an estimate a few preliminaries are needed.
\begin{lemma} \label{1f}
The following hold on $\Omega^1\!\Hh$
\begin{itemize}
\item[(i)] $[\Delta_{\Hh},\di_{\Hh}^{\star} ]=- 2 \di_{\Hh}^{\star}  \circ \p $
\item[(ii)]  $[\Delta_{\Hh}+\C,I_a]=0$
\item[(iii)] $[\Delta_{\Hh},\p]=0$.
\end{itemize}
\end{lemma}
\begin{proof}
(i) follows by dualising part (iii) in Lemma \ref{h-diff}. Part (ii) is proved by the same cone argument as the one used in dimension $7$ in \cite[Lemma 4.2]{NS}. The claim in (iii) follows easily by applying the Lie derivatives $\L_{\xi_a}$ in (ii), since the Casimir operator $\C$ is $\su(2)$ invariant.
\end{proof}
Next we describe some of the features of the representation of $\su(2)$ on $\Omega^1\!\Hh$, in analogy with the case of functions treated at the beginning of this section. Define the $\su(2)$ invariant subspaces $\Omega^1_k{\Hh}:=\Omega^1\!\Hh \cap \ker(\p-k)$ for $k \in \mathbb{Z}$ and denote by $\Omega_{B}^1\Hh$ the space of basic, i.e. $\su(2)$ invariant, horizontal $1$-forms. By (iii) in Lemma \ref{1f} the weighted spaces $\Omega^1_{k}\!\Hh, k \in \mathbb{Z}$ are preserved by $\Delta_{\Hh}$. Observe that 
the identity $\C=\p^2-2\p$ easily entails  
\begin{equation*}
\begin{split}
&\Omega^1\!\Hh \cap \ker(\C-m(m+2))=\Omega^1_{-m}\!\Hh \oplus\  \Omega^1\!_{m+2}\!\Hh, \ m \in \mathbb{N}\\
&\Omega^1_{1}\!\Hh=\Omega^1_{2}\!\Hh=0, \ \Omega^1_0\!\Hh=\Omega_{B}^1\!\Hh .
\end{split}
\end{equation*}
\begin{rem} \label{v-eq}
Note that some of the spaces $\Omega^1_k\Hh$ may vanish identically. The group $\Aut(M,g)$ acts on the spaces
$\Omega^1\Hh \cap \ker(\C-m(m+2))$; in case $\SO(3) \subseteq 
\Aut(M,g)$ we know that the representations $(\rho_l,V_l)$ of $\su(2)$ do not induce representations 
of $\SO(3)$ hence it is easy to conclude that $\Omega^1_{k}\Hh=0$ for $k$ odd. This type of argument 
will also be used in the proof of Theorem \ref{main3}.
\end{rem}
The behaviour of the spaces $\Omega^1_k\Hh$ under the action of $I_1,I_2,I_3$ is captured in the following 
\begin{lemma}\label{sp1}
For any $m \in \mathbb N$ the space $ \Omega^1_{-m}\!\Hh \oplus\ \Omega_{m+4}^1\!\Hh  $ is $\sp(1)$-invariant, that is invariant under the complex structures $I_a$, for
$a=1,2,3$. 
In particular, whenever $\alpha \in  \Omega^1_{-m}\!\Hh $, the component of $I_a \alpha$ in $ \Omega^1_{-m}\!\Hh $  is given by
$
(I_a \alpha)_{-m} = \tfrac{1}{m+2} \L_{\xi_a} \alpha
$
and similarly
$
(I_a \beta)_{m+4} = - \tfrac{1}{m+2} \L_{\xi_a} \beta
$
for $\beta \in \Omega^1_{m+4}$.
Moreover, the map
$$\mathbf{a}_{-m} : \Omega^1_{-m}\!\Hh  \rightarrow \Omega^1_{m+4}\!\Hh, \ \mathbf{a}_{-m} (\alpha) :=  (I_a \alpha)_{m+4}
$$
is injective and satisfies 
\begin{equation} \label{D-a}
\Delta_{\Hh} \circ   \mathbf{a}_{-m}  =  \mathbf{a}_{-m} \circ (\Delta_{\Hh} - 4m -8) .
\end{equation}
\end{lemma}
\begin{proof}
All algebraic claims follow directly from the identities in Lemma \ref{l1s2}. See also \cite[Lemma 6.1]{NS} for fully detailed proofs in dimension $7$. To prove \eqref{D-a} pick $\alpha \in \Omega^1_{-m}\!\Hh$ which thus belongs to $\ker(\C-m(m+2))$. Taking into account the 
$\sp(1)$-invariance of $\Delta_{\Hh}+\C$ in Lemma \ref{1f},(ii) shows that 
\begin{equation*}
\begin{split}
\Delta_{\Hh} I_a\alpha=&(\Delta_{\Hh}+\C)I_a\alpha-\C(I_a\alpha)=I_a(\Delta_{\Hh}+\C)\alpha-\C(I_a\alpha)\\
=&I_a\Delta_{\Hh}\alpha+
m(m+2)I_a\alpha-\C(I_a\alpha).
\end{split}
 \end{equation*}
Because the operators $\Delta_{\Hh}$ and $\C$ commute with $\p$ the claim is proved by projecting onto $\Omega^1_{m+4}\!\Hh$.
\end{proof}
Our second main estimate for the scalar sub-Laplacian, refining that in Proposition \ref{est-f1}, is contained in the following 
\begin{propn}\label{est2}
The scalar sub-Laplacian  satisfies the lower bound
\begin{equation*}
\Delta_{\Hh} \geq (n+1)(4m+8)
\end{equation*}
on $C^{\infty}_mM \cap \ker(\Delta_{\Hh}-4nm)^{\perp}$ where $m \in \mathbb{N}$.
\end{propn}
\begin{proof}
Consider $f \in C^{\infty}_m M \cap \ker (\Delta_{\Hh} - \lambda)$ where $\lambda \neq 4nm$. Due to the commutator identities $[\Delta_{\Hh}, \p] = 0 $ and $[\Delta_{\Hh}, \d_{\Hh}]  = 2 \p \circ \d_{\Hh}$ (see Lemma \ref{h-diff}, (iii)) we get $\Delta_{\Hh} (\d_{\Hh} f)_{-m} = (\lambda - 2m)(\d_{\Hh} f)_{-m}$. Further on 
the form 
$$\gamma:=\mathbf{a}_{-m}(\d_{\Hh}\!f)_{-m}=(I_a(\d_{\Hh}\!f)_{-m})_{m+4} \in \Omega^1_{m+4}\!\Hh$$ satisfies 
$ \Delta_{\Hh}\gamma=(\lambda-6m-8)\gamma
$
according to \eqref{D-a}. Since the map $\mathbf{a}_{-m}$ is injective the form $\gamma$ cannot vanish identically, otherwise 
$(\d_{\Hh}\!f)_{-m}=0$ hence $f \in \ker(\Delta_{\Hh}-4nm)$ by (i) in Proposition \ref{est-f1}, which is a contradiction. The claim thus follows from the estimate in Proposition \ref{Lco-2} applied to the eigenform $\gamma$ above.
 \end{proof}
This entails, as an immediate consequence, a Lichnerowicz-Obata type estimate for the scalar basic Laplacian, which we recall, acts on $\su(2)$-invariants functions. 
\begin{cor}\label{basic1} Let $M^{4n+3}$ be compact and equipped with a $3$-Sasaki structure $(g,\xi)$. The scalar basic Laplacian $\Delta_B:=\Delta^g_{\vert C_B^{\infty}M}: C_B^{\infty}M \to C_B^{\infty}M$ satisfies  
$$ \lambda_1(\Delta_B) \geq 8(n+1).
$$
If $g$ does not have constant sectional curvature this estimate is strict.
\end{cor}
\begin{proof}
As the estimate follows from Proposition \ref{est2} with $m=0$ we only need to show that equality does not occur if $g$ does not have 
constant sectional curvature. As this fact is not needed in this section we postpone its proof until Remark \ref{2nd-pf}, (ii).
\end{proof}
When $n=1$ this result follows, as explained 
in \cite{NS}, from the Obata type estimates for the basic Laplacian of a Riemannian foliation in \cite{LeeR}. When $(g,\xi)$ is regular 
the estimate has been proved in \cite[Corollary 3.5]{LeB}, see also \cite{AMP}. The lower bound on the basic Laplacian above complements that in Proposition \ref{est-f1}, which is not covered when $m=0$. Based on the results in this section, we can prove the first part of Theorem \ref{main1}, as formulated below.
\begin{thm} \label{thm1-half}Assuming that $g$ does not have constant sectional curvature the scalar sub-Laplacian satisfies the lower bound $\lambda_1(\Delta_{\Hh}) \geq 8n$. The limiting eigenspace satisfies
$\ker(\Delta_{\Hh}-8n) \subseteq C_2^{\infty}M$.
\end{thm}
\begin{proof}
The spaces $C^{\infty}_mM $ where $m \in \mathbb{N}$ are preserved by the sub-Laplacian and moreover their 
direct sum is dense in $C^{\infty}M$. It is therefore enough to collect the previously obtained estimates 
for the first non-zero eigenvalue of $\Delta_{\Hh}$ acting on each of these spaces and then take the infimum with $m \in \mathbb{N}$. Indeed, by Corollary \ref{basic1} we know that $\Delta_{\Hh} \geq 8(n+1)>8n$ on non-constant basic functions. 
At the same time, Proposition \ref{est-f1} ensures that $\Delta_{\Hh} \geq 4nm \geq 8n $ on $C_m^{\infty}M$, provided that $m \geq 2$.
When $m=1$, the space $\ker(\Delta_{\Hh}-4n)$ vanishes by part (ii) in the same proposition. 
Hence we have $\Delta_{\Hh} \geq 12(n+1)>8n$ on $C_1^{\infty}M$ by Proposition \ref{est2}.
This allows concluding 
that $\lambda_1(\Delta_{\Hh}) \geq 8n$, with equality attained on functions in  $C_2^{\infty}M$.
\end{proof}
\subsection{Spectral description of the tri-moment map } \label{tri-sectn}
To fully prove Theorem \ref{main1} there remains to describe the limiting eigenspace $\ker(\Delta_{\Hh}-8n)$. 
First we recall a few well known facts related to isometries and automorphisms of $3$-Sasaki structures. The Lie algebra of infinitesimal automorphisms of a $3$-Sasaki structure $(g,\xi)$ on a compact manifold  $M^{4n+3}$ is defined according to 
\begin{equation*}
\gg := \mathfrak{aut} (M,\xi^a) = \{ X \in \Gamma(TM) : \L_X \xi^a = 0, a=1,2,3\}.
\end{equation*}
Throughout this section we assume that $g$ {\it{does not}} have constant sectional curvature, so that the space of Killing fields splits as 
\begin{equation} \label{iso-1}
\mathfrak{aut}(M,g) = \spa\{\xi_1, \xi_2, \xi_3\} \, \oplus \, \gg  
\end{equation}
see e.g. \cite{BoGa} for a proof. 
The {\it  tri-moment map } $\mu : \mathfrak{g} \rightarrow C^\infty (M, \mathbb R^3)$ is defined
by $\mu= (\mu_1, \mu_2, \mu_3)^T$ with $\mu_a(X):= g(\xi_a, X).$  
The tri-moment map satisfies the $\su(2)$-transformation rules 
$\L_{\xi_a}\mu_a(X)=0, \ \L_{\xi_a}\mu_b(X)=-\L_{\xi_b}\mu_a(X)=2\mu_c(X)$ with cyclic permutation on $abc$.
Equivalently, 
\begin{equation} \label{inf1}
\mu(X) \in C^{\infty}_{\pi_1}(M,U_1)
\end{equation}
in the notation of Section \ref{lb1}. Here we recall that $(\pi_1,U_1 \cong \bbR^3)$ is the adjoint representation of the Lie algebra $\su(2)$. In addition 
\begin{eqnarray}
&& I_1\di_{\Hh}\!\mu_1(X)=I_2\di_{\Hh}\mu_2(X)=I_3\di_{\Hh}\mu_3(X).  \label{inf2} 
\end{eqnarray}
Conversely, any function  $f \in C^\infty_{\pi_1}(M, U_1)$ satisfying \eqref{inf2} gives rise to an
infinitesimal automorphism via $X= \sum_a f_a \xi_a - \tfrac12 I_1 \, \grad_{\Hh}\, f_1 \in \gg$ with $f= \mu(X)$. In particular the map 
$\mu:\gg \to  C^\infty_{\pi_1}(M, U_1)$ is injective. The proofs of these facts are local in nature and based on the $\su(2)$-commutator relations in \eqref{su2} and Cartan's formula. 

The first step towards characterising the tri-moment map by its spectral properties hinges on a specific fact from Sasaki-Einstein geometry which we now explain. Consider the Sasaki-Einstein structure $(g,\xi_a)$ where $a \in \{1,2,3\}$ together with the Lie subalgebra 
$\mathfrak{h}_a:=\{X \in \mathfrak{aut}(M,g) : [X,\xi_a]=0 \ \mbox{and} \ \int_M \xi^a(X)\vol=0\} \subseteq \mathfrak{aut}(M,g).$ As it is the case for any Sasaki-Einstein structure the moment map induces an isomorphism 
\begin{equation} \label{iso-S}
\tau_a:\mathfrak{h}_a \to \ker(\Delta_{B_a}-8(n+1)), \  X \mapsto g(X,\xi_a),
\end{equation}
 with inverse given by $f \mapsto f\xi_a-\frac{1}{2}\tilde{I}_a\grad f$ in the notation of Section \ref{op-id}. With respect to the Sasaki-Einstein structure $(g,\xi_a)$ we consider the basic Laplacian 
 $$\Delta_{B_a}:=\Delta^g_{\vert C^{\infty}_{B_a}M}:
C^{\infty}_{B_a}M \to C^{\infty}_{B_a}M$$ 
where the space of basic functions $C^{\infty}_{B_a}M:=\{f \in C^{\infty}M : \L_{\xi_a}f=0\}$. For a proof of \eqref{iso-S} see \cite[Theorem 11.3.1]{BoGa} which deals with the more general case of extremal Sasaki metrics; for a cone view of the isomorphism in \eqref{iso-S} see the proof of \cite[Theorem 2.14., (iii)]{HS}. By analogy with K\"ahler geometry functions in 
$\ker(\Delta_{B_a}-8(n+1))$ will be called Killing potentials. Because $g$ is assumed not to have constant sectional curvature the splitting \eqref{iso-1} ensures that $\mathfrak{h}_a=\gg$ and thus $\tau_a=\mu_a$.

\begin{propn} \label{tri-m2}
The tri-moment map induces a linear isomorphism 
$$\mu:\g \to \ker(\Delta_{\Hh}-8n) \cap C^{\infty}_{\pi_1}(M,U_1).$$
\end{propn}
\begin{proof}
Letting $X \in \gg$ we have $\mu(X) \in C^{\infty}_{\pi_1}(M,U_1)$ by \eqref{inf1}; in addition, combining  \eqref{inf1} and \eqref{inf2} with the definition of $\p$ yields $\d_{\Hh} \mu_a(X) \in \Omega^1_4\! \Hh$.
Then 
$$\Delta_{\Hh} \mu_a(X) = \di^{\star}_{\Hh} \di_{\Hh} \mu_a(X) = \tfrac14  \di^{\star}_{\Hh}  \p \di_{\Hh} \mu_a(X) = n \C \mu_a(X) = 8n \mu_a(X)$$ by taking into account Lemma \ref{h-diff},(ii). It follows that 
$\mu$ is well defined as stated; it is also injective, due to the isomorphisms in \eqref{iso-S}. To prove surjectivity for $\mu$, pick $f=(f_1,f_2,f_3) \in \ker(\Delta_{\Hh}-8n) \cap C^{\infty}_{\pi_1}(M,U_1)$. Then $f_a \in C^{\infty}_{B_a}M$ and 
$$ \Delta_{B_a}f=\Delta_{\Hh}f-\L_{\xi_2}^2f-\L_{\xi_3}^2f=8(n+1)f$$ thus \eqref{iso-S} provides a triple $X_a \in \gg$ with 
$\mu_a(X_a)=f_a$ for $a=1,2,3$. Vertical Lie differentiation of these equalities, combined with having $f \in C^{\infty}_{\pi_1}(M,U_1)$ and 
\eqref{inf1} shows that $\mu_a(X_b)=f_a$ for $a \neq b \in \{1,2,3\}$.
In particular $\mu_a(X_b-X_c)=0$ with cyclic permutations on $abc$, which forces $X_1=X_2=X_3$ by using, for instance, the isomorphisms in \eqref{iso-S}. The surjectivity of $\mu$ is thus proved.
\end{proof}
At this stage a few remarks are in order.
\begin{rem} \label{2nd-pf}
\begin{itemize}
\item[(i)] When $(M^{4n+3},g,\xi)$ is regular a proof of Proposition \ref{tri-m2} can be also extracted from \cite{AMP} after going through the identifications in Remark \ref{reg-case}; it mainly uses the K\"ahler geometry of the twistor space of the smooth quaternion-K\"ahler quotient $M\slash \mathcal{F}$. 
\item[(ii)] The basic Laplacian of the Sasaki-Einstein structure $(g,\xi_a)$ satisfies  the estimate $\Delta_{B_a}\geq 8(n+1)$ according to 
\cite[Theorem 5.1]{F}. The lower bound stated in that reference is $4(n+1)$ which is due to horizontal rescaling in the metric $g$. As 
$C^{\infty}_BM \subseteq C^{\infty}_{B_a}M $ we recover the estimate $\Delta_B \geq 8(n+1)$ from Corollary \ref{basic1}. This estimate is strict: since $C^{\infty}_BM \cap \ker(\Delta_B-8(n+1)) \subseteq C^{\infty}_{B_1}M \cap \ker(\Delta_{B_1}-8(n+1))$ a function $f$ in the former space satisfies, by \eqref{iso-S}, $\mu_1(X)=f$ for $X \in \gg$.  As $\C\mu_1(X)=8\mu_1(X)$ and $f$ is $\su(2)$ invariant it follows that $f=0$. This argument completes the proof of Corollary \ref{basic1}.
\end{itemize}
\end{rem}
The second part in Theorem \ref{main1},  i.e. the full  description of the limiting eigenspace  $\ker(\Delta_{\Hh}-8n)$ 
can now be dealt with.
\begin{thm} \label{lim-f}
The map 
$$\widetilde{\mu} : \gg \oplus \gg \oplus \gg \to \ker(\Delta_{\Hh}-8n), \ \ \tilde{\mu}(X_1,X_2,X_3):=\sum_a\mu_a(X_a)$$ 
is a linear isomorphism.
\end{thm}
\begin{proof}
Direct computation based on \eqref{inf1} leads to 
$\L_{\xi_a}\L_{\xi_b}\widetilde{\mu}(X_1,X_2,X_3)=4\mu_b(X_a)$ for $a \neq b \in \{1,2,3\}$. 
Having $(X_1,X_2,X_3) \in \ker \widetilde{\mu}$ thus forces $\mu_b(X_a)=0$ when $a \neq b \in \{1,2,3\}$. By \eqref{iso-S} 
it follows that $X_1=X_2=X_3=0$ hence
$\widetilde{\mu}$ is injective. To show that $\widetilde{\mu}$ is surjective pick $f \in \ker(\Delta_{\Hh}-8n)$ and recall that $f \in C_2^{\infty}M$ by Theorem \ref{thm1-half}, in particular $\C f = 8 f$. But $8$ is
the Casimir eigenvalue of the adjoint representation $(\pi_1,U_1)$ of $\su(2)$. Hence, the finite dimensional  $\su(2)$ representation space 
$\ker (\Delta_{\Hh} - 8n)$ decomposes into a finite sum of $U_1$'s. Choose a basis $\{f_1, f_2, f_3\}$ in one of these summands such that 
$\L_af_a=0, \L_{\xi_a}f_b=-\L_{\xi_b}f_a=2f_c$, that is the triple $(f_1, f_2, f_3)$ belongs to $ C^\infty_{\pi_1} (M,U_1)$. As Proposition \ref{tri-m2} ensures that $(f_1,f_2,f_3)=\mu(X)$ with $X\in \gg$, any linear combination 
$\lambda_1f_1+\lambda_2f_2+\lambda_3f_3=\mu_1(\lambda_1X)+\mu_2(\lambda_2X)+\mu_3(\lambda_3X)$.  Repeating this argument 
for all copies of $U_1$ possibly occurring in $\ker (\Delta_{\Hh}-8n)$ proves surjectivity for $\widetilde{\mu}$.
\end{proof}
$\\$
{\bf{Proof of Theorem \ref{main1}}} Follows by combining the estimate in Theorem \ref{thm1-half} and Theorem \ref{lim-f} above.\\

Further on, we denote with $\mathcal H_l(\RM^3)$ the space of harmonic polynomials on $\bbR^3$ which are homogeneous of degree $l$. 
As a step towards proving Theorem \ref{main3}, we prove that using these polynomials many more eigenfunctions for $\Delta_{\Hh}$ can be generated from the components of the tri-moment map.

\begin{propn} \label{poly}
For any pair $(P,X)$ in $\mathcal H_l(\RM^3) \times \gg$
we have 
$$P(\mu(X)):=P(\mu_1(X),\mu_2(X),\mu_3(X)) \in \ker(\Delta_{\Hh}-8ln) \cap C^\infty_{2l} M.$$
\end{propn}
\begin{proof}
Whenever $\varphi \in C^{\infty}\bbR^3$ we indicate with $\varphi_{x_i}, \varphi_{x_jx_k}$ its partial derivatives up to second order and with $\partial_R\varphi=\sum_a x_a\varphi_{x_a}$ its radial derivative. By calculus 
\begin{equation*}
\frac{1}{4}\C \varphi(\mu(X))=\Vert \mu(X) \Vert^2 (\Delta_{\bbR^3}\varphi)\mu(X)+(\partial_R^2\varphi+\partial_R\varphi)\mu(X)
\end{equation*}
after taking into account the $\su(2)$ invariance properties of the tri-moment map as encoded in $\mu(X) \in C^{\infty}_{\pi_1}(M,U_1)$. For instances when 
$\varphi=P \in \mathcal H_l(\RM^3)$ we have $\Delta_{\bbR^3}\varphi=0$ and  $\partial_R\varphi=l\varphi$ thus $P(\mu(X))$ belongs to 
$C_{2l}^{\infty}M$. To prove the eigenvalue equation for $P(\mu(X))$ record that the horizontal gradients of $\mu_1(X),\mu_2(X), \mu_3(X)$ are orthogonal and have equal norm; this follows from \eqref{inf2}. Again by the chain rule  
\begin{equation*}
\begin{split}
\Delta_{\Hh}\ \varphi(\mu(X))=&\sum_a \varphi_{x_a}(\mu(X))\Delta_{\Hh}\mu_a(X)-\sum_{1 \leq a,b \leq 3}\varphi_{x_ax_b}(\mu(X))\langle \di_{\Hh}\mu_a(X), \di_{\Hh}\mu_b(X)\rangle\\
=&8n (\partial_R\varphi)\mu(X)+\Vert \di_{\Hh}\mu_1(X) \Vert^2(\Delta_{\bbR^3}\varphi)\mu(X).
\end{split}
\end{equation*}
Letting $\varphi=P \in \mathcal H_l(\RM^3)$ ensures that $P(\mu(X))$ belongs to $ \ker(\Delta_{\Hh}-8nl)$.
\end{proof}
$\\$
{\bf{Proof of Theorem \ref{main3}.}} Recall that $n=1$ in that setup. By assumption we know that $\Aut(M,g)=\Aut(M,\xi) \times \SO(3)$, thus the group 
$\SO(3)$ acts on $C^{\infty}M$. Since the representations $(\rho_l,V_l)$ of $\su(2)$ do not induce representations for the group $\SO(3)$, we conclude that $C_{m}^{\infty}M=0$ for $m\in 2\mathbb{N}+1$. On $C_0^{\infty}M$ we have $\Delta_{\Hh}>16$ by Corollary \ref{basic1}, since 
$g$ does not have constant sectional curvature. On $C_2^{\infty}M \cap \ker(\Delta_{\Hh}-8)^{\perp}$ Proposition \ref{est2} ensures that $\Delta_{\Hh}\geq 32>16$. Finally, on $C_m^{\infty}M, m \geq 4$ we have that $\Delta_{\Hh} \geq 4m\geq 16$, with equality when $m=4$, by using again Proposition \ref{est-f1}. These considerations show that $\lambda_2(\Delta_{\Hh})=16$ with limiting eigenspace $\ker(\Delta_{\Hh}-16) \subseteq 
C^{\infty}_4M$. Since $\H_2(\bbR^3) \cong \Sym^2_0\bbR^3$ Proposition \ref{poly} 
shows,  after polarisation, that the image of the linear map $\Sym^2\g \otimes \Sym^2_0\bbR^3 \to C^{\infty}M$ 
given by 
\begin{equation} \label{map14}
S \otimes \beta \mapsto \tr (\beta(\mu \circ S,\mu)+\beta(\mu, \mu \circ S))
\end{equation} is contained in 
$\ker(\Delta_{\Hh}-16) \cap C^{\infty}_4M$. Here the trace is taken with respect to the $L^2$ inner product on $\g$. This concludes 
the proof of Theorem \ref{main3}.\\ 

As promised in the introduction we show that the map in \eqref{map14} is non-zero on specific elements as follows.
\begin{exo} \label{ex1}Let $X \in \gg$. Purely algebraic arguments based on the $\su(2)$ invariance properties of $\mu$ in \eqref{inf1} show, after integration by parts, that the system 
$ \{\mu_1(X)\mu_2(X),\mu_2(X)\mu_3(X),\mu_3(X)\mu_1(X) \}$  
is $L^2$ orthogonal in $C_4^{\infty}M$ and also that 
$$ \int_M \mu_1^2(X)\mu_2^2(X)\vol= \int_M \mu_2^2(X)\mu_3^2(X)\vol=\int_M \mu_3^2(X)\mu_2^2(X)\vol=\frac{1}{3}\int_M\mu_a^4(X)\vol
$$
for $a \in \{1,2,3\}$. In addition the system $\{\mu_1^2(X),\mu_2^2(X),\mu_3^2(X)\}$ is $L^2$ orthogonal to $ \{\mu_1(X)\mu_2(X),\mu_2(X)\mu_3(X),\mu_3(X)\mu_1(X) \}$. A direct $L^2$ orthogonality argument based on these facts shows that the element $\lambda_1\mu_2(X)\mu_3(X)+
\lambda_2\mu_3(X)\mu_1(X)+\lambda_3\mu_1(X)\mu_2(X)$ in 
$\Ker(\Delta_{\Hh}-16n) \cap C_4^{\infty}M$, where $\lambda=(\lambda_1,\lambda_2,\lambda_3) \in \bbR^3$, is non-zero unless $\lambda=0$.
\end{exo}
\subsection{Lower bounds for Riemannian Laplacians} \label{crmk} 
As a by-product of the techniques used to prove Theorem \ref{main1} we derive estimates for the first non-zero eigenvalue $\lambda_1(\Delta^{g_T})$ of the scalar Riemannian Laplacian $\Delta^{g_T}$ where $g_T:=T^2\sum \limits_{a=1}^3\xi^a \otimes \xi^a+g_{\Hh}, T>0$ is the canonical variation of the $3$-Sasaki metric, $g=g_1$. 
\begin{thm} \label{RL-T}
Let $(M^{4n+3},g,\xi)$ be a compact $3$-Sasaki manifold such that $g$ does not have constant sectional curvature. The scalar Riemannian Laplacian of $g_T$ satisfies 
\begin{equation*}
\lambda_1(\Delta^{g_T}) \geq 8(n+T^{-2})
\end{equation*} 
for all $T>0$. The limiting eigenspace $\ker(\Delta^{g_T}-8(n+T^{-2}))$ is isomorphic to $\g \oplus \g \oplus \g$ via the map $\widetilde{\mu}$ from Theorem \ref{lim-f}.
\end{thm}
\begin{proof}
The argument is entirely similar to the one in the proof of Theorem \ref{main1}. First we use the relation $\Delta^{g_T}=\Delta_{\Hh}+T^{-2}\C$ and then we combine the estimates for $\Delta_{\Hh}$ in Propositions \ref{est-f1} and \ref{est2} with the estimate for the basic Laplacian in Corollary \ref{basic1}.
\end{proof}

This estimate is particularly relevant for the Einstein metrics in the canonical variation $g_T$ which are obtained for $T=1$ and 
$T=\frac{1}{\sqrt{2n+3}}$. When $T=1$ it reads 
\begin{equation} \label{estRR1}
\lambda_1(\Delta^g) \geq 8(n+1).
\end{equation}
In this case the cone metric is hyperk\"ahler, in particular Ricci flat K\"ahler. In the next section we will show how a different proof of
\eqref{estRR1} can be obtained, after a few additional arguments, from the gap theorem for Ricci flat K\"ahler cones.
\begin{rem} \label{indep}
The estimate $\lambda_1(\Delta_{\Hh}) \geq 8n$ from Theorem \ref{main1} does not follow from \eqref{estRR1}. Indeed $\Delta_{\Hh}=\Delta^{g}-\C$ and the operator $\C\geq 0$ is unbounded. Nonetheless, once Theorem \ref{RL-T} is established for all $T>0$,  
the aforementioned estimate can be recovered by letting $T \to \infty$.
\end{rem}
When $T=\frac{1}{\sqrt{2n+3}}$ the geometry of the cone metric is very different, namely hyperk\"ahler with torsion, a fact which essentially follows from  \cite{FFU-V}. Also note that $g_{\q}$ is no longer Ricci flat, whereas this is the case for the metric $a^2(\di\!r)^2+r^2g_T$ for the correct choice of $a \in \bbR$. For this value of $T$ we get the estimate 
\begin{equation*} \label{estR1}
\lambda_1(\Delta^{g_T}) \geq 24(n+1)
\end{equation*}
which we believe to be new. 
\begin{rem} \label{stab-half}
Using the well known formula $\scal^{g_T}=16n(n+2)+6T^{-2}-12nT^2$ for $T>0$ (see \cite{BoGa}) one can check that for both these values of $T$ 
we have $\lambda_1(\Delta^{g_T})>2E_T$ with $E_T$ the Einstein constant of $g_T$. This estimate is half of the requirements, see \cite{Cao-He}[Corollary 1.3], needed to ensure $\nu$-stability for the Einstein metric $g_T$; the second requirement, that is having the Lichnerowicz Laplacian bounded below by $2E_T$ on $TT$-tensors, is however violated as it follows from \cite[Chapter 9, Fig.9.72]{Besse}.
\end{rem}
\section{Gap phenomena and holomorphic functions } \label{sectn4}
\subsection{Gap theorem for hyperk\"ahler cones} \label{gap-hk}
Recall that functions $f\in C^{\infty}CM $ are called $\tau$-homogeneous  provided that $\L_{r\partial_r}f=\tau f$ for some $\tau \in \bbR$ which will be called the growth rate of $f$. Equivalently the function $\tilde{f}:=r^{-\tau}f$ belongs to $C^{\infty}M$. The additional condition $\Delta^{g_{\q}}f=0$ translates, by using the comparison formula \eqref{lapl-c0}, into the eigenvalue 
equation 
\begin{equation} \label{eqn-57}
\Delta^g\tilde{f}=\tau(\tau+4n+2)\tilde{f}.
\end{equation}
This explains how eigenfunctions for $\Delta^g$ lift to harmonic functions on the cone. In addition, estimates for the scalar Laplacian, are equivalent to growth rate estimates; for instance having $\lambda_1(\Delta^g) \geq 8(n+1)$ as in \eqref{estRR1} is equivalent to 
$\tau \geq 2$. To put these observations into perspective recall the following result \cite{CoH,HS} which deals with growth rate estimates in the more general case of Ricci flat K\"ahler cones, that is metric cones over compact Einstein Sasaki manifolds. 
\begin{thm} \label{gap-HS} Let $CN$ be a Ricci flat, non-flat, K\"ahler cone and let $f:CN \to \bbR$ be harmonic and $\tau$-homogeneous. Then
 $\tau>1$ and 
\begin{itemize}
\item[(i)] in case $\tau<2$ the function $f$ is the real part of a holomorphic function; in particular $\L_{\zeta}^2f=-\tau^2f$ where 
$\zeta$ is the Reeb vector field of the cone
\item[(ii)] in case $\tau=2$, the function $f=r^2f_1+f_2$ where $f_1 \in C^{\infty}N$ is a Killing potential and $f_2$ is the real part of a $2$-homogeneous holomorphic function on $CN$.
\end{itemize}
\end{thm}
\begin{proof}
Follows by combining Lemma 2.13 and Theorem 2.14 in \cite{HS}.
\end{proof}
The main goal of this section is to obtain an improvement of Theorem \ref{gap-HS} in the case of hyperk\"ahler cones. 
This will be done using direct arguments based on the estimates in the previous sections. First, we recall the well-known fact that hyperk\"ahler cones are quasi-regular, in the following sense.
\begin{lemma} \label{reg-3}
Let $M$ be compact and equipped with a $3$-Sasaki structure $(g,\xi)$. The Sasaki-Einstein structure $(g,\xi_a)$ is quasi-regular in the sense that $\xi_a$ integrates to an effective action of 
the circle $\bbR\slash 2\pi r\mathbb{Z}$ with $r=1$ or $r=2$. In addition, the eigenvalues of $\L^2_{\xi_1}$ acting on $C^{\infty}M$ are $-r^2k^2$ with $ k \in \mathbb{Z}$.
\end{lemma}
\begin{proof}
The first part of the claim follows by exponentiating within $\Sp(1)$ or $\SO(3)$, according to the dichotomy in \eqref{dich}. The second part follows by 
Fourier decomposition in $C^{\infty}M$ w.r.t. the action of the circle  $\bbR\slash 2\pi r\mathbb{Z}$.
\end{proof}
\begin{rem} \label{proof21}
Let $(M^{4n+3},g,\xi)$ be compact and $3$-Sasaki. Apply Theorem \ref{gap-HS} to the Ricci flat K\"ahler cone of the Sasaki-Einstein structure $(g,\xi_1)$. 
Since the eigenvalues of $\L^2_{\xi_1}$ acting on $C^{\infty}M$ are $-r^2k^2$ with $ k \in \mathbb{Z}$ and $r \in \{1,2\}$ (see Lemma \ref{reg-3}) the growth rate $\tau \in (1,2)$ for real valued functions on $CM$ is prohibited. Thus instances as in 
(i) in Theorem \ref{gap-HS} do not occur, i.e. $\tau \geq 2$ in that theorem and we obtain a second proof for the estimate $\lambda_1(\Delta^g) \geq 8(n+1)$. There remains to relate our description (see Theorem \ref{RL-T}) of the limiting eigenspace to the information in the gap Theorem \ref{gap-HS}; to that extent we essentially 
need to determine holomorphic functions on the cone. This will be done in the next section for arbitrary growth rate, not only for $\tau=2$.
\end{rem}
Below we assume that the $3$-Sasaki metric $g$ on the compact manifold $M$ does not have constant sectional curvature and prove the following 
hyperk\"ahler version of Theorem \ref{gap-HS}.
\begin{thm} \label{gap}
Assume that $f \in C^{\infty}CM$ satisfies $\Delta^{g_{\q}}f=0$ and $\L_{r\partial_r}f=\tau f$. The following hold 
\begin{itemize}
\item[(i)] we have $\tau \geq 2$ with equality when $r^{-2}f \in \IM \widetilde{\mu} \cong \g \oplus \g \oplus \g$
\item[(ii)] if $2 <\tau<3$ the function $f$ is $\su(2)$ invariant, $\L_{\xi_a}f=0$ for $ 1 \leq a \leq 3$
\item[(iii)] if $\tau=3$ we have $r^{-3}f \in C_0^{\infty}M \oplus C_1^{\infty}M \oplus C_3^{\infty}M$.
\end{itemize}
\end{thm}
\begin{proof}
(i) As already discussed we have $\tau \geq 2$ directly from \eqref{estRR1}.
The description of $f$ when $\tau=2$ follows from Theorem \ref{RL-T}.\\
(ii) Then $12n+15> \tau(\tau+4n+2)>8n+8$. Since $\Delta^g$ is $\su(2)$-invariant, it is enough 
to examine which components of $\tilde{f}=r^{-\tau}f$ on the weighted spaces $C_m^{\infty}M$  with $m \in \mathbb{N}$ may occur.  When $m \geq 3$ we have 
$$\Delta^g=\Delta_{\Hh}+\C\geq 4nm+m(m+2) \geq 12n+15$$
by Proposition \ref{est-f1}. Thus $\tilde{f}$ can only have components in $C_m^{\infty}M$ for $m \in \{0,1,2\}$. For $m=1$ the space $\ker(\Delta_{\Hh}-4nm)$ vanishes, see Proposition \ref{est-f1},(ii); for $m=2$ the function $\tilde{f}$ belongs to $C_m^{\infty}M\cap \ker(\Delta_{\Hh}-4nm)^{\perp}$ since its corresponding eigenvalue $\tau(\tau+4n+2)>8n+8>8n$ by \eqref{eqn-57}.
Thus in both cases we can apply Proposition \ref{est2} to get the lower bound 
$$\Delta^g=\Delta_{\Hh}+\C \geq (n+1)(4m+8)+m(m+2) \geq 12n+15$$ on $C_m^{\infty}M\cap \ker(\Delta_{\Hh}-4nm)^{\perp}$ for $m \in \{1,2\}$. 
Since $12n+15> \tau(\tau+4n+2)$ we conclude that $\tilde{f}$ must be invariant, as claimed.\\
(iii) is proved by arguments similar to those in (ii).
\end{proof}
\begin{rem}\label{tau=3}
Under the assumptions in Theorem \ref{gap}, (iii), the component of $r^{-3}f$ on $C_3^{\infty}M$ belongs to $\ker(\Delta_{\Hh}-12n)$ and thus it is further characterised as in Proposition \ref{est-f1}, (i). 
When $\SO(3) \subseteq \Aut(M,g)$ the spaces $C_{m}^{\infty}M$ vanish for $m$ odd (see also the proof of Theorem \ref{main3}), which enables the much stronger conclusion 
$r^{-3}f \in C_0^{\infty}M$.
\end{rem}
\subsection{Holomorphic functions on hyperk\"ahler cones} \label{hol-desc}
Here the aim is -- with motivation partly stemming from Remark \ref{proof21} -- to explore the holomorphic content of the limiting eigenspaces 
$\ker(\Delta_{\Hh}-4nm) \cap C_m^{\infty}M$. Indicate with $\mathrm{Hol}_{\tau}(CM,J_1)$ the space of functions $f:CM \to \bbC$ which are holomorphic with respect to $J_1$, that is $J_1\di\!f=i\di\!f$, and also $\tau$-homogeneous, for some $\tau \in \bbR$. It is a simple observation that functions $f \in \mathrm{Hol}_{\tau}(CM,J_1)$ satisfy 
\begin{equation} \label{holf11}
\L_{\xi_1}\tilde{f}=-i\tau \tilde{f} \ \mbox{and} \ \L_{\xi_3}\tilde{f}=i\L_{\xi_2}\tilde{f}.
\end{equation}
This follows by evaluating the holomorphy requirement on $f$ on the vector fields $\partial_r$ respectively $\xi_2$, whilst using \eqref{alg-J}. Denoting 
$$\mathscr{E}_m:=\ker(\Delta_{\Hh}-4nm) \cap C_m^{\infty}M \cap \ker(\L_{\xi_1}^2+m^2).$$
for $m \in \mathbb{N}$, we prove that 
\begin{propn} \label{hol-cone}
The following hold 
\begin{itemize}
\item[(i)] $\mathrm{Hol}_{\tau}(CM,J_1)=0$ unless $\tau \in \mathbb{N}$
\item[(ii)] for any $m \in \mathbb{N}$ the map 
$$ f \in \mathrm{Hol}_{m}(CM,J_1) \mapsto r^{-m}\re(f) \in \mathscr{E}_m
$$
is a linear isomorphism 
\item[(iii)] when $\tau=2$ we have $\g \oplus \g \cong \mathrm{Hol}_2(CM,J_1)$ via the map
$$(X,Y) \mapsto r^2(\mu_2+i\mu_3)X-ir^2(\mu_2+i\mu_3)Y.$$
\end{itemize}
\end{propn}
\begin{proof}
(i) According to \eqref{holf11} the function $\tilde{f}$ satisfies $\L^2_{\xi_1}\tilde{f}=-\tau^2\tilde{f}$.
Since $(g,\xi_1)$ is quasi-regular by Lemma \ref{reg-3}
the claim is proved. \\
(ii) We first show that $\tilde{f}$, where $f \in \Hol_m(CM,J_1)$, belongs to $C_m^{\infty}(M,\bbC)$. This follows, after also taking \eqref{holf11} into account, from
\begin{equation*}
\begin{split}
\C \tilde{f}=&m^2\tilde{f}-\L_{\xi_2}^2\tilde{f}-\L_{\xi_3}^2\tilde{f}=m^2\tilde{f}+i(\L_{\xi_2}\L_{\xi_3}-\L_{\xi_3}\L_{\xi_2})\tilde{f}\\
=&m^2\tilde{f}+i\L_{[\xi_2,\xi_3]}\tilde{f}=m^2\tilde{f}+2i\L_{\xi_1}\tilde{f}=(m^2+2m)\tilde{f}.
\end{split}
\end{equation*}
Since $f$ is holomorphic, it is in particular harmonic, thus $\Delta_{\Hh}\tilde{f}=(\Delta^g-\C)\tilde{f}=4nm\tilde{f}$ after using \eqref{eqn-57}. Writing 
$\tilde{f}=u+iv$ thus ensures that $u$ belongs to $\mathscr{E}_m$.

Conversely, letting $u$ belong to $\mathscr{E}_m$ the function $F:=u+\frac{i}{m}\L_{\xi_1}u$ satisfies 
$$\L_{\xi_1}F=-im F \ \mbox{and} \  \C F=m(m+2)F.$$ 
We will show that $f:=r^mF \in \mathrm{Hol}_m(CM,J_1)$ in two steps, as follows. First we compute 
\begin{equation*}
\begin{split}
\Vert \L_{\xi_3}F-i\L_{\xi_2}F\Vert^2_{L^2}=&\Vert \L_{\xi_3}F\Vert^2_{L^2}+\Vert \L_{\xi_2}F\Vert^2_{L^2}-i\langle \L_{\xi_2}F,\L_{\xi_3}F\rangle_{L^2}+i\langle \L_{\xi_3}F,\L_{\xi_2}F\rangle_{L^2}\\
=&-\langle (\L_{\xi_2}^2+\L_{\xi_3}^2)F,F\rangle_{L^2}+i\langle F, [\L_{\xi_2},\L_{\xi_3}]F \rangle_{L^2}\\
=& \langle (\C+\L_{\xi_1}^2)F,F\rangle_{L^2}+2i
\langle F,\L_{\xi_1}F\rangle_{L^2}=0.
\end{split}
\end{equation*}
Hence $\L_{\xi_3}F=i\L_{\xi_2}F$ which amounts to the vanishing of $J_1\di\!f-i\di\!f$ on the distribution $ \spa\{\partial_r\}\oplus \Vv \subseteq TCM$. There remains to prove the vanishing of the horizontal component in $J_1\di\!f-i\di\!f$; equivalently we have to show that 
$I_1\di_{\Hh}F=i\di_{\Hh}F$. This follows from computing the norm 
\begin{equation*}
\begin{split}
\Vert I_1\di_{\Hh}F-i\di_{\Hh}F\Vert^2_{L^2}=&2\Vert \di_{\Hh}\!F\Vert^2_{L^2}-i\langle \di_{\Hh}\!F,I_1\di_{\Hh}F\rangle_{L^2}+i
\langle I_1\di_{\Hh}\!F, \di_{\Hh}F\rangle_{L^2}\\
=&2 \langle \Delta_{\Hh}F,F \rangle_{L^2}-8nm \Vert F\Vert^2_{L^2}=0
\end{split}
\end{equation*} 
after using that $\di_{\Hh}^{\star}\!I_1 \dH\!F=-4n\L_{\xi_1}F=4nmiF$ as granted by Lemma \ref{h-diff}, (i).\\
(iii) Due to the isomorphism in (ii) it is enough to describe the space $\mathscr{E}_2$. According to Theorem \ref{lim-f} any 
$u \in \mathscr{E}_2$ can be written as $u=\mu_1(X_1)+\mu_2(X_2)+\mu_3(X_3)$ with $X_a \in \gg$. The $\su(2)$-invariance properties of the tri-moment map in \eqref{inf1} grant that $\L_{\xi_1}^2u=-4(\mu_2(X_2)+\mu_3(X_3))$ hence having $u \in \ker(\L_{\xi_1}^2+4)$ forces 
$\mu_1(X_1)=0$. Because the components of tri-moment map are injective by \eqref{iso-S}, we get $X_1=0$. In particular we see that $\L_{\xi_1}u=2(\mu_3(X_2)-\mu_2(X_3))$. The claim follows by recording that the inverse of the isomorphism in (ii) reads 
$u \mapsto r^2(u+\frac{i}{2}\L_{\xi_1}u)$.
\end{proof}
\begin{rem} \label{reg-hol}
When $(M,g)$ is regular and simply connected it is well-known that $M$ can be identified with the sphere bundle in $L_Z^{-1}$ in such a way that the induced circle action on $M$ is tangent to $\xi_1$.
Here $Z$ is the twistor space of $Q=M\slash\mathcal{F}$ and $L_Z$ is the contact line bundle of $Z$. It follows that $\Hol_m(CM,J_1)$ is naturally identified with the space of holomorphic sections of the positive complex line bundle 
$L_Z^{m}$. 
\end{rem}

\end{document}